\documentclass[11pt]{amsart}
\usepackage{amsthm,amsmath,stmaryrd,bbm,geometry,color}
\usepackage{amssymb}
\usepackage[english]{babel}
\usepackage[utf8]{inputenc}
\usepackage{graphicx}
\usepackage{verbatim}
\usepackage{enumitem}
\usepackage{mathtools}
\usepackage[titletoc]{appendix}
\usepackage{bm}
\usepackage{textcomp}
\usepackage{amsaddr}
\allowdisplaybreaks

\usepackage[dvipsnames]{xcolor}
\usepackage[hyperindex=true,frenchlinks=true,colorlinks=true,
citecolor=Mahogany,linkcolor=Red,urlcolor=Tan,linktocpage]{hyperref}

\usepackage{tikz}
\usetikzlibrary{shapes}
\usetikzlibrary{fit}
\usetikzlibrary{decorations.pathmorphing}

\title[Law of large numbers for strictly stationary Banach-valued random 
fields]{Weak and strong law of large numbers for strictly stationary 
Banach-valued random fields}


\date{\today}

\address[$\dagger$]{Institut de Recherche Mathématique Avancée
UMR 7501, Université de Strasbourg and\\ CNRS
7 rue René Descartes
67000 Strasbourg, France
}
\email{dgiraudo@unistra.fr}
\author{Davide Giraudo}
\numberwithin{equation}{section}
\renewcommand{\leq}{\leqslant}
\renewcommand{\geq}{\geqslant}

\newtheorem{Theorem}{Theorem}[section]
\newtheorem{Proposition}[Theorem]{Proposition}
\newtheorem{Lemma}[Theorem]{Lemma}
\newtheorem{Definition}[Theorem]{Definition}
\newtheorem{Corollary}[Theorem]{Corollary}

\theoremstyle{remark}

\tikzstyle{Vertex}=[circle,draw=LimeGreen!80,fill=LimeGreen!8,
inner sep=1pt,minimum size=2mm,line width=1pt,font=\scriptsize]
\tikzstyle{Node}=[Vertex,draw=RoyalBlue!80,fill=RoyalBlue!8,inner sep=1.5pt]
\tikzstyle{Leaf}=[rectangle,draw=Black!70,fill=Black!16,
inner sep=0pt,minimum size=1mm,line width=1.25pt]
\tikzstyle{Edge}=[Maroon!80,cap=round,line width=1pt]
\tikzstyle{Mark1}=[draw=BrickRed!80,fill=BrickRed!8]
\tikzstyle{Mark2}=[draw=BurntOrange!80,fill=BurntOrange!8]
\tikzstyle{EdgeRew}=[->,RedOrange!80,cap=round,thick]

\newcommand{\intent}[1]{\llbracket #1\rrbracket}
\newcommand{\Aca}{\mathcal{A}}
\newcommand{\Bca}{\mathcal{B}}

\newcommand{\Fca}{\mathcal{F}}
\newcommand{\Gca}{\mathcal{G}}

\newcommand{\iid}{i.i.d.\ }

\newcommand{\B}{\mathbb{B}}
\newcommand \ens[1]{\left\{ #1\right\}}

\newcommand{\gk}{\gr{k}}

\newcommand \R{\mathbb R}

\newcommand \N{\mathbb N}
\newcommand \PP{\mathbb P}
\newcommand{\el}{\mathbb L}
\newcommand{\E}[1]{\mathbb E\left[#1\right]}
\newcommand{\pr}[1]{\left(#1\right)}

\newcommand \Z{\mathbb Z}

\newcommand \abs[1]{\left|#1\right|}
\newcommand \eps{\varepsilon}

\newcommand{\norm}[1]{\left\lVert #1 \right\rVert}
\newcommand{\gr}[1]{\bm{#1}}
\newcommand{\imd}{\preccurlyeq}
\newcommand{\smd}{\succcurlyeq}
\newcommand{\grn}{\gr{n}}
\newcommand{\grN}{\gr{N}}
\newcommand{\gru}{\gr{u}}
\newcommand{\gri}{\gr{i}}
\newcommand{\grj}{\gr{j}}
\newcommand{\grk}{\gr{k}}

\newcommand{\ind}[1]{\mathbf{1}_{#1}}

\newcommand{\ent}[1]{\left\lfloor #1\right\rfloor}

\newcommand{\card}[1]{\operatorname{Card}\pr{#1}}
\newcommand{\Id}{\operatorname{Id}}
\begin{document}

\begin{abstract}
 In this paper, we investigate the law of large numbers for strictly stationary random fields, that is, we provide sufficient conditions on
 the moments and the dependence of the random field in order to guarantee the almost sure convergence to $0$ and the convergence in $\mathbb L^p$ of partials sums over squares or rectangles of $\mathbb Z^d$. Approximation by multi-indexed martingales as well as
 by $m$-dependent random fields are investigated. Applications to 
 functions of $d$-independent Bernoulli shifts and to functionals of 
\iid random fields are also provided.
\end{abstract}
\maketitle

\section{Introduction}

The law of large numbers is one of the most fundamental theorems in probability theory and statistics. It states that if $\pr{X_i}_{i\geq 1}$ is an i.i.d.\ sequence such that $\E{\abs{X_1}}<\infty$, then $n^{-1}\sum_{i=1}^nX_i$ converges to $\E{X_1}$ in the almost sure sense. A more general version, called Marcinkiewicz strong law of large numbers, states that if $1\leq p<2$ and $\E{\abs{X_1}^p}<\infty$, then $n^{-1/p}\sum_{i=1}^n\pr{X_i-\E{X_i}}$ converges to $0$ almost surely. Several authors investigated then the case of 
dependent sequences, see for instance 
\cite{MR1334177,MR2111727,MR1785247,MR0799152,MR2743029}.  

We are interested  in analoguous results for collections of random variables 
indexed by $\Z^d$.
Throughtout the paper, we will denote for a positive integer $d$
the set $\ens{1,\dots,d}$ by $\intent{1,d}$, for 
$\gri=\pr{i_\ell}_{\ell\in\intent{1,d}}$ and 
$\grj=\pr{j_\ell}_{\ell\in\intent{1,d}}$, $\gri\imd\grj$ means that 
for each $\ell\in\intent{1,d}$, $i_\ell\leq j_\ell$. We will write $\gr{1}$ 
for the element of $\Z^d$ whose coordinates are all equal to one and for 
$\grn=\pr{n_\ell}_{\ell\in\intent{1,d}}$, 
$\abs{\grn}=\prod_{\ell=1}^d n_\ell$.
We are looking for conditions on the dependence and the moments of 
the centered random field $\pr{X_{\gri}}_{\gri\in\Z^d}$ taking values in a 
Banach space $\pr{\B,\norm{\cdot}_{\B}}$ in order to guarantee 
the convergence of 
\begin{equation}\label{eq:general_statement_LLN_rect}
 \frac 1{\abs{\grn}}\norm{\sum_{\gr{1}\imd\gri\imd\grn}X_{\gri}}_{\B}
\end{equation}
to $0$ in the almost sure sense or in $\el^p$ for some $p$ in a range that 
depends on the Banach space $\B$, as one of the coordinates of $\grn$ goes to 
infinity. We will also consider the almost sure convergence to $0$ of the sum 
of squares, namely
\begin{equation}\label{eq:general_statement_LLN_squares}
 \frac 1{n^{d/p}}\norm{\sum_{\gr{1}\imd\gr{i}\imd n\gr{1}}X_{\gri} }_{\B}\to 0
\end{equation}
as $n$ goes to infinity. 

In the \iid and $\B=\R$ case, \cite{MR494431} provided a necessary and 
sufficient condition for the 
convergence the term involved in \eqref{eq:general_statement_LLN_rect} as all 
the coordinates of $\grn$ go to infinity when $\B=\R$, namely, that 
\begin{equation}
 \E{\abs{X_{\gr{0}} }^p\pr{\log\pr{1+\abs{X_{\gr{0}} }}}^{d-1}  }<\infty.
\end{equation}
Still for independent random variables, a necessary and sufficient condition 
has been found in \cite{MR0871253}, Theorem~2.2, when the summation is done 
over 
sets that are not necessarily rectangles.

Under the assumption that the random field $\pr{X_{\gri}}_{\gri\in\Z^d}$ is 
i.i.d.\ and satisfies a condition which is stronger than $\E{\abs{X_1}}<\infty$
but weaker than $\E{\abs{X_1}\pr{\log\pr{1+\abs{X_1}}  }^{d-1}  }$, a strong law 
of large numbers for weighted sums over rectangle have been established in 
\cite{MR3892761}.

Under the independence but without the identical distribution assumption, a law 
of large number and convergence rates have been obtained in \cite{MR4515464}. 
A result of the same spirit for pairwise independent random fields has been 
obtained in \cite{MR1944146}.

Assymetric law of large numbers, that is, with the normalization
$\prod_{\ell\in\intent{1,d}}n_\ell^{\alpha_\ell}$ with potentially different 
exponents $\alpha_\ell$, for negatively associated random fields has been 
investigated in \cite{MR3211166}.
In Theorem~3.2 of \cite{MR2608813}, a strong law of large numbers
for martingale random fields has been established. On one hand, there condition 
on the martingale property is more restrictive than orthomartingales (see 
Definition~\ref{dfn:def_orthomartingale}). On the 
other hand, we put restriction
on the filtration, namely commutativity, while Dung and Duy Tien do not.
When applied to identically distributed random fields, convergence
\eqref{eq:convergence_presque_sure} holds under a slightly stronger
moment assumption since they need $\E{\norm{D_{\gr{0}}}^{p+\delta}_{\B}}$ to
be finite for some positive $\delta$. This result was improved in
\cite{MR2970301}, Theorem~3.3, where a similar result under the
same martingale assumption as in \cite{MR2608813}, but with
the optimal moment assumption 
$\E{\norm{D_{\gr{0}}}_{\B}^p\pr{\log\pr{1+\norm{D_{\gr{0}}}_{\B}}}^{d-1}}
<\infty$
. An other type of martingale difference random fields
were investigated in \cite{MR3706696}, which does not seem to be directly 
comparable with orthomartingales. The case of real-valued orthomartingales have 
been treated in \cite{MR2794415}, but it is not easy to compare with our 
conditions because of the condition (4.2) in the aforementioned paper where 
some 
convergence rates on a conditional maximum are required. An other result 
concerning orthomartingale difference random fields which are not necessarily 
identically distributed or stochastically dominated has been studied in 
\cite{MR3069895}.
When $p=1$, the convergence \eqref{eq:convergence_presque_sure} has been
shown in \cite{MR3452228} in the context of pairwise independent and identically 
distributed random fields.

In this paper, we are interested in the strong law of large for random fields, 
that is, collections of random variables indexed by $d$-uples of integers and 
the summation over $\ens{1,\dots,n}$
is replaced by rectangles of $\Z^d$. We will assume that the involved random 
variables 
take their values in a separable Banach space $\pr{\B,\norm{\cdot}_{\B}}$.

We will use essentially two approaches: a first one 
by approximating via multi-indexed martingales and a second one by 
approximation by $m$-dependent random fields, giving different ranges of 
application. 

The paper is organized as follows. In 
Section~\ref{sec:LGN_orthomartingale_approx}, we establish results 
on the Marcinkiewicz law of large numbers for stationary random fields 
using an approximation by multi-indexed martingales. We first complement the 
already obtained statement for the strong law of large numbers on rectangles by 
providing sufficient conditions for the strong law of large numbers on squares 
and the convergence in $\el^p$. Then we provide results for stationary random 
fields by approximating by 
such martingale differences random fields. This method was performed in order 
to derive weak invariance principles 
\cite{MR3869881,MR3222815,MR3646441,MR4486233,MR4477977,MR3798239}, quenched 
invariance principles \cite{MR4125956,MR4166203,reding:hal-03246736} and 
the bounded law of the iterated logarithms \cite{MR4186670}.

In Section~\ref{sec:LLN_shifts},  the results obtained by 
orthomartingale approximation are used in order to derive a Marcinkiewicz 
strong law of large numbers for random fields that can be expressed as a Hölder 
continuous function of $d$ mutually independent sequences expressable as 
functions of \iid sequences.
The condition involves the exponent of Hölder regularity, the dimension $d$ and 
the dependence coefficient of each involved sequence.

In Section~\ref{sec:lgn_Bernoulli random fields}, we investigate the 
aforementioned laws of large numbers for random fields expressable as a 
functional of an \iid random field. We propose an application to a Hölder 
continuous funtional of a Banach valued linear random field. The condition is 
written in terms of the exponent of Hölder regularity, the moments of the 
innovations and the operator norm of the coefficients.

Section~\ref{sec:proofs} is devoted to the proofs of the results of 
Sections~\ref{sec:LGN_orthomartingale_approx}, \ref{sec:LLN_shifts} 
and \ref{sec:lgn_Bernoulli random fields}. The needed auxiliary results are 
grouped in the Appendix.

\section{Weak and strong law of large numbers via orthomartingale 
approximation}\label{sec:LGN_orthomartingale_approx}

\subsection{The orthormartingale case}
\label{subsec:lln_ortho}

The notion of multi-indexed martingale requires the notion of multi-indexed 
filtration. We will also require the filtration to be commuting in the 
following sense.
\begin{Definition}\label{dfn:commutante}
 We say that a collection of $\sigma$-algebras $\pr{\Fca_{\gri}}_{\gri\in\Z^d}$
is a completely commuting filtration if
\begin{enumerate}
 \item for each $\gri,\grj\in\Z^d$ such that $\gri\imd\grj$,
$\Fca_{\gri}\subset\Fca_{\grj}$ and
 \item for each $Y\in\mathbb L^1$ and each $\gri,\grj\in\Z^d$,
 \begin{equation}\label{eq:def_filtration_commutante}
  \E{\E{Y\mid\Fca_{\gri}}\mid\Fca_{\grj}}=
  \E{Y\mid\Fca_{\min\ens{\gri,\grj}}},
 \end{equation}
 where $\min\ens{\gri,\grj}$ is the element of $\Z^d$ defined as the
 coordinatewise minimum of $\gri$ and $\grj$, that is,
 $\min\ens{\gri,\grj}=\pr{\min\ens{i_\ell,j_\ell}}_{\ell=1}^d$.
\end{enumerate}

\end{Definition}

Let us give two examples of commuting filtrations.
\begin{Proposition}\label{prop:exemple_filtra_commutantes}
 \begin{enumerate}
 \item If $\pr{\eps_{\gru}}_{\gru\in\Z^d}$ is i.i.d., then defining
 $\Fca_{\gri}=\sigma\pr{\eps_{\gru},\gru\in\Z^d,\gru\imd\gri}$, the filtration
 $\pr{\Fca_{\gri}}_{\gri\in\Z^d}$ is completely commuting.
 \item Suppose that $\pr{\Fca^{\pr{\ell}}_{i_\ell}
}_{i_\ell \in \Z   }$, $1\leq \ell\leq d$, are filtrations on a probability space $\pr{\Omega,\Fca,\PP}$. Suppose
that for each $i_1,\dots,i_d\in\Z$, the
$\sigma$-algebras $\Fca^{\pr{\ell}}_{i_\ell}$, $1\leq \ell\leq d$, are independent. Let
$\Fca_{\gri}=\bigvee_{\ell=1}^d\Fca^{\pr{\ell}}_{\gr{i_\ell}}$. Then
$\pr{\Fca_{\gri}}_{\gri\in\Z^d}$ is completely commuting.
\end{enumerate}
\end{Proposition}

Both examples where introduced in Section 1 of \cite{MR0420845}, but without
proof. The first item is a direct consequence of Proposition~2 p. 1693 of
\cite{MR3222815}.

We are now in position to define orthomartingale
martingale difference random field, which allows to exploit the martingale
property in every
direction. To formize this, we need to denote by
$e_{\gr{\ell}}$, $\ell\in\intent{1,d}$,  the element of $\Z^d$ whose
$\ell$-th coordinate is $1$ and all the others are zero.

\begin{Definition}\label{dfn:def_orthomartingale}
 Let $\pr{X_{\gri}}_{\gri\in \Z^d}$ be a random field taking values in a
separable Banach space $\pr{\B,\norm{\cdot}_{\B}}$. We say that
$\pr{X_{\gri}}_{\gri\in \Z^d}$ is an orthomartingale
martingale difference random field with respect to the completely commuting
filtration
$\pr{\Fca_{\gri}}_{\gri\in\Z^d}$ if for each $\gri\in\Z^d$,
$\norm{X_{\gri}}_{\B}$ is integrable, $X_{\gri}$ is $\Fca_{\gri}$-measurable 
and for
each $\ell\in\intent{1,d}$, $\E{X_{\gri}\mid \Fca_{\gri-\gr{e_\ell}}}=0$.
\end{Definition}

Such a definition is very convenient because summation on a rectangular region
of $\Z^d$ can be treated with martingale properties when summing on a fixed
coordinate. 

The proof of the law of large number usually rest on satisfactory moment 
inequalities for martingales. Therefore, we will work will smooth Banach spaces 
in the following sense.

\begin{Definition}
 Let $\pr{\B,\norm{\cdot}_{\B}}$ be a separable Banach space. We say that $\B$
is $r$-smooth for $1<r\leq 2$ if there exists an equivalent norm
$\norm{\cdot}'_{\B}$ on $\B$ such that
\begin{equation} 
\sup_{t>0}\sup_{x,y\in\B,\norm{x}'_{\B}=\norm{y}'_{\B}=1,}\frac{\norm{x+ty}'_{\B
}+\norm{x-ty}'_{\B}-2}{
t^r}<\infty.
\end{equation}

\end{Definition}

For example, if $\mu$ is $\sigma$-finite on the Borel $\sigma$-algebra of $\R$, 
then $\mathbb L^p\pr{\R,\mu}$ is $\min\ens{p,2}$-smooth.
Moreover, a separable Hilbert space is $2$-smooth.

Given a random variable $X$ taking values in a separable Banach space 
$\pr{\B,\norm{\cdot}_{\B}}$ and $p\geq 1$, $q\geq 0$, we define 
\begin{equation}\label{eq:def_norme_el_p}
\norm{X}_{\B,p}:=\pr{\E{\norm{X}_{\B}^p}}^{1/p}\mbox{ and }
\end{equation}
\begin{equation}\label{eq:def_norme_el_pq}
\norm{X}_{\B,p,q}:=\norm{\norm{X}_{\B}}_{p,q},
\end{equation}
where, for a real valued random variable $Y$, 
\begin{equation}\label{eq:def_varphi_pq}
\norm{Y}_{p,q}=\inf\ens{\lambda>0,
\varphi_{p,q}\pr{\frac{\abs{Y}}{\lambda}}\leq 1},
\varphi_{p,q}\pr{t}=t^p\pr{1+ \ind{t\geq 1}\ln t }^q.
\end{equation}
Note that $\norm{X}_{\B,p,0}=\norm{X}_{\B,p}$.
Denote by $\mathbb L_{p,q}$ the space of random variables
$Y$ such that $\E{\varphi_{p,q}\pr{\norm{Y}_{\B}}}<\infty$.
 Conditions of the form 
 \begin{equation}\label{eq:cond_Lpq_fini}
\E{\norm{D_{\gr{0}}}_{\B}^p
\pr{\log\pr{1+\norm{D_{\gr{0}}}_{\B}}
}^{q}}<\infty 
\end{equation}
for some $p$ and $q$ are usual in the context of random fields. 
For instance, the quenched functional central limit theorem on rectangles 
for a strictly stationary orthomartingale difference random field
requires \eqref{eq:cond_Lpq_fini} for $p=2$ and $q=d-1$ (see \cite{MR4125956}) 
and for the bounded law of the iterated logarithms, \eqref{eq:cond_Lpq_fini} 
for $p=2$ and $q=2d-2$ (see \cite{MR4186670}). Therefore, the condition 
on the moments can be expressed with the help of the norm $\norm{\cdot}_{p,q}$.

We also define 
\begin{equation}
\norm{X}_{\B,p,w}:=\sup_{A:\PP\pr{A}>0}\PP\pr{A}^{-1+1/p}
\E{\norm{X}_{\B}\ind{A}}.
\end{equation} 
When $\B=\R$, we shall simply write $\norm{X}_{p}$, $\norm{X}_{p,q}$ and 
 $\norm{X}_{p,w}$ respectively. 
Note that there exists constants $c_p$ and $c'_p$ such that for each 
random variable $X$, 
\begin{equation}\label{eq:comparaison_norm_Lp_faible}
c_p\pr{\sup_{t>0}t^p\PP\pr{\norm{X}_{\B}>t}}^{1/p}\leq \norm{X}_{\B,p,w}
\leq c'_p\pr{\sup_{t>0}t^p\PP\pr{\norm{X}_{\B}>t}}^{1/p}.
\end{equation}

For $\grn=\pr{n_\ell}_{\ell=1}^d\in\N^d$, we define 
$\gr{2^n}=\pr{2^{n_\ell}}_{\ell =1}^d$  
and $\max\grn=\max_{1\leq \ell\leq d}n_\ell$ 
and we recall that $\abs{\grn}=\prod_{\ell=1}^dn_\ell$.

The following  result has been obtained for the convergence of 
normalized partial sums on rectangles of an orthomartingale difference random 
field. 
\begin{Theorem}[Theorem~2.2 in 
\cite{giraudo2022deviation}, law of large numbers on rectangles]
\label{thm:loi_forte_des_grands_nombres_orthomartingale}
  Let $\pr{\B,\norm{\cdot}_{\B}}$ be a separable
  $r$-smooth Banach space for some $r\in
(1,2]$,  $1<p<r$ and $d\in\N$. There exists a constant
  $K_{p,d,\B}$ such that the following holds: if
$\pr{D_{\gri}}_{\gr{i}\in\Z^d}$ is an identically distributed
   orthomartingale difference random field such
  that  $\norm{D_{\gr{0}}}_{\B,p,d-1} <\infty$,
  then for all positive $t$, the following inequality holds
  \begin{equation}\label{eq:control_of_sum_PAn}
  \sum_{\gr{N}\in\N^d}\PP\pr{ \abs{\gr{2^{N}}}^{-1/p}
  \max_{\gr{1}\imd\grn\imd \gr{2^N}}\norm{S_{\grn}}_{\B}
  >t}
  \leq K_{p,d}\E{\varphi_{p,d-1}\pr{\frac{\norm{D_{\gr{0}}}_{\B}}t }},
  \end{equation}
  where $S_{\grn}=\sum_{\gr{1}\imd\gri\imd\grn}D_{\gri}$.
  In particular, for some constant $C_{\B,d,p}$ depending only on $\B$, $d$ 
  and $p$,
  \begin{equation}\label{eq:control_weak_Lp_norm_maximum}
\norm{\sup_{\gr{n}\smd\gr{1}}\frac{\norm{S_{\gr{n}}}_{\B}
}{\abs{\gr {n } } ^ { 1/p } } }_{\B,p,w}
  \leq C_{\B,d,p} \norm{D_{\gr{0}}}_{\B,p,d-1}
  \end{equation}
  and the following convergence holds:
  \begin{equation}\label{eq:convergence_presque_sure}
  \lim_{N\to \infty  }\sup_{\grn\smd\gr{1}, 
\max\grn \geq N}\frac{\norm{S_{\gr{n}}}_{\B}
}{\abs{\gr{n}}^{1/p}}
=0\mbox{
  almost surely.}
  \end{equation}
\end{Theorem}
Such a result was know for $d=1$ (see \cite{MR0668549} and Proposition~2.1 in 
\cite{MR3322323}).

When we consider the summation over squares, that is, sets 
of the form $\intent{1,n}^d$, where $n$ is an integer bigger than one, it turns out that we can embed the orthomartingale into an one-dimensional martingale. As a consequence, only finite moments of order $p$ are required. 
\begin{Theorem}[Law of large numbers 
on squares]\label{thm:loi_forte_grands_nombres_carres}
  Let $\pr{\B,\norm{\cdot}_{\B}}$ be a separable
  $r$-smooth Banach space for some $r\in
(1,2]$,  $1<p<r$ and $d\in\N$. There exists a constant
  $C_{\B,d,p}$ such that the following holds: if
$\pr{D_{\gri}}_{\gr{i}\in\Z^d}$ is an identically distributed
   orthomartingale difference random field such
  that  $\norm{D_{\gr{0}}}_{\B}\in\mathbb L_{p}$, then the following inequality holds:
  \begin{equation}\label{eq:controle_fonction_maximale_carres_orthomartingale}
  \norm{\sup_{n\geq 1}\frac{\norm{S_{n\gr{1}} }_{\B}}{n^{d/p}}}_{p,w}\leq 
C_{\B,d,p}\norm{D_{\gr{0}}}_{\B,p}.
  \end{equation}
  Moreover, the following convergence holds:
  \begin{equation}
  \lim_{n\to\infty}\frac{\norm{S_{n\gr{1}}}_{\B}}{n^{d/p}}=0\mbox{ almost surely}.
  \end{equation}
\end{Theorem}

For convergence in $\el^p$, we have the following result.

\begin{Theorem}[Convergence in 
$\mathbb L^p$]\label{thm:loi_faible_des_grands_nombres_orthomartingale}

  Let $\pr{\B,\norm{\cdot}_{\B}}$ be a separable
  $r$-smooth Banach space for some $r\in
(1,2]$,  $1<p<r$ and $d\in\N$. If
$\pr{D_{\gri}}_{\gr{i}\in\Z^d}$ is an identically distributed
   orthomartingale difference random field such
  that  $\norm{D_{\gr{0}}}_{\B}\in\mathbb L^{p}$, then 
  \begin{equation}\label{eq:marcink_wlln}
  \lim_{\max\grN\to\infty}\frac 1{\abs{\grN}^{1/p}}\norm{\max_{\gr{1}\imd\grn\imd\grN}\norm{S_{\grn}  }_{\B}}_p=0.
  \end{equation}
\end{Theorem}
We point out that the convergence in \eqref{eq:marcink_wlln} is as $\max\grN\to\infty$, in other words, we require that only one of the coordinates of $\grN$ 
goes to infinity. Moreover, unlike in the case of almost sure convergence,  the consideration of squares instead of rectangles would not give a less restrictive condition.

\subsection{Orthomartingale approximation}\label{subsec:orthomart_approx}

In this section, we assume that the random field $\pr{X_{\gri}}_{\gri\in\Z^d}$ is of the form 
\begin{equation}\label{eq:setting}
X_{\gri}=X_{\gr{0}}\circ T^{\gri},\quad \Fca_{\gri}=T^{-\gri}\Fca_{\gr{0}}
\end{equation} where $T^{\gri}\colon\Omega\to\Omega$ is such that $T^{\gri}\circ T^{\grj}
=T^{\gri+\grj}$ for each $\gri,\grj\in\Z^d$.  

For example, one can consider the 
case where $\Omega=\R^{\Z^d}$ endowed with the product measure of
a probability measure $\mu$ and $T^{\gri}$ is the shift operator given by 
$T^{\gri}\pr{\pr{x_{\gr{k}}}_{k\in\Z^d}}=\pr{x_{\gr{k}+\gri}}_{k\in\Z^d}$.

An other example is the following: take probability spaces $\pr{\Omega_\ell, 
\Aca_\ell,\mu_\ell}$ and let $\Omega=\prod_{\ell=1}^d\Omega_\ell$ endowed with the product $\sigma$-algebra and the product measure $\mu$. Let $T_\ell\colon\Omega_\ell\to\Omega_\ell$ be a bijective bi-measure preserving map and let $\Fca_0^{\pr{\ell}}$ be a sub-$\sigma$-algebra of $\Aca_\ell$ such that
$T_\ell\Fca_0^{\pr{\ell}}\subset \Fca_0^{\pr{\ell}}$. Then define 
$\Fca_{\gri}:=\bigvee_{\ell=1}^dT_\ell^{i_{\ell}}\Fca_0^{\pr{\ell}}$.

For $\gri\in\Z^d$, we define the map $U^{\gri}$ by $U^{\gri}\pr{f}\pr{\omega}=f\pr{T^{\gri}\omega}$.
In order to extend the  results of Subsection~\ref{subsec:lln_ortho} to a 
larger class of stationary random 
fields, we define the projection operator 
\begin{equation}
P_{\gk}\pr{Y}:=\sum_{I\subset\intent{1,d}}
\pr{-1}^{\abs{I}}\E{Y\mid \Fca_{\gk-\ind{I}}},
\end{equation}
where $\abs{I}$ denotes the cardinality of $I$ and 
$\gk-\ind{I}=\pr{k_\ell-\ind{\ell\in I}}_{\ell=1}^d$.
When $d=1$, $P_k\pr{Y}=\E{Y\mid \Fca_k}-\E{Y\mid \Fca_{k-1}}$ 
and when $d=2$, 
\begin{equation}
P_{k_1,k_2}\pr{Y}=\E{Y\mid\Fca_{k_1,k_2}}-\E{Y\mid\Fca_{k_1-1,k_2}}
-\E{Y\mid\Fca_{k_1,k_2-1}}+\E{Y\mid\Fca_{k_1-1,k_2-1}}.
\end{equation}

The norm of  such projectors is used in order to measure how far 
a random field from an orthomartingale difference random field is. 
Indeed, if $\pr{D_{\gri}}_{\gri\in\Z^d}$ is an orthomartingale 
difference random field, then 
$P_{\gk}\pr{D_{\gr{0}}}=0$ if $\gk\neq \gr{0}$. Moreover, the fact that 
the filtration is defined with the help of the action $T$ gives, for $\gri$, 
$\gk\in\Z^d$, 
\begin{equation}\label{eq:stationarity_projecteurs}
\norm{P_{\gri+\gk}\pr{X_{\gri}}}_{p,q}=\norm{P_{\gk}\pr{X_{\gr{0}}}}_{p,q}.
\end{equation}

It is tempting to express $X_{\gri}$ as a sum of of projetors, namely, 
$X_{\gri}=\sum_{\gk\in\Z^d}P_{\gk}\pr{X_{\gri}}$, where the sum is understood 
as $\lim_{m\to\infty}\sum_{\gk\in\Z^d,\norm{\gk}_\infty\leq 
m}P_{\gk}\pr{X_{\gri}}$ and the limit in the sense of the 
$\norm{\cdot}_{\B,p,q}$ for some $p$ and $q$. Since for fixed $m$, the sum 
$\sum_{\gk\in\Z^d,\norm{\gk}_\infty\leq 
m}P_{\gk}\pr{X_{\gri}}$ is telescopic, only $2^m$ terms are remaining: one of 
them is $\E{X_{\gri}\mid\Fca_{m\gr{1}}}$ and the other one 
of the form $\E{X_{\gri}\mid \Fca_{m 
\gr{1_I}-\pr{m+1}\gr{1_{\intent{d}\setminus I}}  }}$ for some proper subset 
$I$ of $\intent{1,d}$, where $\gr{1_I}=\sum_{i\in I}\gr{e}_i$. In order to make 
their contribution negligible, we need the following assumptions:
\begin{equation}\label{eq:hyp_reg1}
 \lim_{m\to\infty}\norm{X_{\gr{0}}-\E{X_{\gr{0}}\mid\Fca_{m\gr{1}}}}_{\B,p,q}=0,
\end{equation}
\begin{equation}\label{eq:hyp_reg2}
\forall \ell_0\in\intent{1,d}, \lim_{m\to\infty}\norm{\E{X_{\gr{0}}  \mid 
\Fca_{m\sum_{\ell\in\intent{1,d}\setminus\ens{\ell_0}}\gr{e}_\ell -
m\gr{e}_{\ell_0}  } 
}}_{\B,p,q}=0. 
\end{equation}

The following results give a strong law and convergence in $\el^p$ for stationary random fields. 

\begin{Theorem}[Law of large numbers on 
rectangles]\label{thm:strong_law_large_stationary}
Let $\pr{X_{\gri}}_{\gri\in\Z^d}$ be a strictly stationary 
random field taking values in a separable $r$-smooth Banach space 
$\pr{\B,\norm{\cdot}_{\B}}$.
and let $\pr{\Fca_{\gri}}_{\gri\in\Z^d}$ be a commuting filtration such that 
\eqref{eq:setting} is satisfied. 
Suppose that \eqref{eq:hyp_reg1} and \eqref{eq:hyp_reg2} hold with $q=d-1$ and 
that 
\begin{equation}
\sum_{\gk\in\Z^d}\norm{P_{\gk}\pr{X_{\gr{0}}}}_{\B,p,d-1}<\infty.
\end{equation}
Then 
  \begin{equation}\label{eq:convergence_presque_sure_cas_stationnaire}
  \lim_{N\to \infty  }\sup_{\grn\smd\gr{1}, 
\max\grn \geq N}\frac 1{\abs{\grn}^{1/p}}\norm{\sum_{\gr{1}\imd\gri\imd\grn}X_{\gri} }  _{\B}
=0\mbox{
  almost surely.}
  \end{equation}
\end{Theorem}

The corresponding result for sums over squares reads as follows.
\begin{Theorem}[Law of large 
numbers on squares]\label{thm:strong_law_large_stationary_squares}
Let $\pr{X_{\gri}}_{\gri\in\Z^d}$ be a strictly stationary 
random field taking values in a separable $r$-smooth Banach space 
$\pr{\B,\norm{\cdot}_{\B}}$.
and let $\pr{\Fca_{\gri}}_{\gri\in\Z^d}$ be a commuting filtration such that 
\eqref{eq:setting} is satisfied. 
Suppose that \eqref{eq:hyp_reg1} and \eqref{eq:hyp_reg2} hold with $q=0$ and 
that 
\begin{equation}
\sum_{\gk\in\Z^d}\norm{P_{\gk}\pr{X_{\gr{0}}}}_{\B,p}<\infty.
\end{equation}
Then 
  \begin{equation} 
  \lim_{n\to \infty  } \frac 
1{n^{d/p}}\norm{\sum_{\gr{1}\imd\gri\imd n\gr{1}}X_{\gri} }  _{\B}
=0\mbox{
  almost surely.}
  \end{equation}
\end{Theorem}

A similar result can be formulated for the convergence in $\el^p$.

\begin{Theorem}[Convergence in $\mathbb L^p$]\label{thm:Lp_law_large_stationary}
Let $\pr{X_{\gri}}_{\gri\in\Z^d}$ be a strictly stationary 
random field taking values in a separable $r$-smooth Banach space 
$\pr{\B,\norm{\cdot}_{\B}}$, $1\leq p<r$,
and let $\pr{\Fca_{\gri}}_{\gri\in\Z^d}$ be a commuting filtration such that 
\eqref{eq:setting} is satisfied. 
Suppose that \eqref{eq:hyp_reg1} and \eqref{eq:hyp_reg2} hold with $q=0$ and 
that
\begin{equation}\label{eq:Hannan_el_p}
\sum_{\gk\in\Z^d}\norm{P_{\gk}\pr{X_{\gr{0}}}}_{\B,p}<\infty.
\end{equation}
Then 
  \begin{equation}\label{eq:Lp_law_large_stationary}
  \lim_{N\to \infty  }\sup_{\grn\smd\gr{1}, 
\max\grn \geq N}\frac 1{\abs{\grn}^{1/p}}\norm{\sum_{\gr{1}\imd 
  \gri\imd\grn}X_{\gri} }_{\B,p}=0.
  \end{equation}
\end{Theorem}

\section{Weak and strong law of large numbers for functions of independent Bernouilli shifts}\label{sec:LLN_shifts}

In this section, we will provide a Marcinkiewicz strong law of large numbers and convergence in $\mathbb{L}p\pr{\B}$ of normalized partials sums of random fields having the form 
\begin{equation}\label{eq:def_fonction_de_d_suites}
X_{\gri}=g\pr{ f_1\pr{ \pr{\eps^{\pr{1}}_{i_1-u_1}}_{u_1\in\Z} },\dots,
f_d\pr{ \pr{\eps^{\pr{d}}_{i_d-u_d}}_{u_d\in\Z} }   },
\end{equation}
where $g\colon\R^d\to\B$ is Hölder continuous with exponent $\alpha$, that is, there exists a constant $C$ such that for each $x_1,\dots,x_d,y_1,\dots,y_d\in\R$, 
\begin{equation}\label{eq:Holder_cont_g}
\norm{g\pr{x_1,\dots,x_d}-g\pr{y_1,\dots,y_d}}_{\B}\leq C\sum_{\ell=1}^d
\abs{x_\ell-y_\ell}^\alpha,
\end{equation}
$\pr{\eps^{\pr{\ell}}_{u_\ell}}_{u_\ell\in\Z}$ are mutually independent i.i.d.\ sequences 
and $f_1,\dots,f_d$ are measurable functions defined on $\R^{\Z^d}$ and taking values in $\R$.
 
The random fields defined via \eqref{eq:def_fonction_de_d_suites} 
are strictly stationary  and the filtration $\pr{\Fca_{\gri}}_{\gri\in\Z^d}$ given by 
\begin{equation}
\Fca_{\gri}=\sigma\pr{\eps_{u_1}^{\pr{1}},u_1\leq i_1,\dots,
\eps_{u_d}^{\pr{d}},u_d\leq i_d}
\end{equation}
is commuting. We would like to apply the results of Subsection~\ref{subsec:orthomart_approx}. However, the assumption 
that for each $\ell\in\intent{1,d}$ and $j_k$, $k\in\intent{1,d}\setminus\ens{\ell}$, 
$\E{X_{\gri}\mid \bigcap_{j_\ell\in\Z} \Fca_{\grj}}=0$ may not be satisfied in most cases. Indeed, let $\B=\R$,  $g\pr{x_1,x_2}=x_1+x_2$ , the maps $f_1,f_2$ are projections on the coordinate of index zero and $\pr{\eps_{u_1}^{\pr{1}}}_{u_1\in\Z}$, $\pr{\eps_{u_2}^{\pr{2}}}_{u_2\in\Z}$ are centered independent i.i.d.\ sequences, then $\E{X_{i_1,i_2}\mid \bigcap_{j_2\in\Z}
\Fca_{i_1,j_2}}=\eps_{i_1}^{\pr{1}}$. 

To overcome this problem, we define for a non-empty subset $J$ of $ \intent{1,d}$ the $\sigma$-algebra 
\begin{equation}
\Gca_J:=\sigma\pr{\eps^{\pr{j}}_{u_j},u_j\in\Z,j\in J},
\end{equation}
and $\Gca_{\emptyset}$ is the trivial $\sigma$-algebra. We then define 
\begin{equation}
X_{\gri}^I:=\sum_{J\subset I}\pr{-1}^{\abs{I}+\abs{J}}\E{X_{\gri}\mid\Gca_J}.
\end{equation}
Notice that for each $J\subset I$, $\E{X_{\gri}\mid\Gca_J}
=\E{X_{\sum_{\ell\in I}i_\ell\gr{e}_\ell}\mid\Gca_J}$ hence the coordinates of 
$\gri$ which do not belong to $I$ do not play any role. For this reason, we 
write
\begin{equation}\label{eq:def_X_gri_I_d_shifts}
 X_{\gri_{I}}^I=\sum_{J\subset I}\pr{-1}^{\abs{I}+\abs{J}}\E{X_{\gri}\mid\Gca_J}, \quad \gri_I=\sum_{\ell\in
I}i_\ell\gr{e}_\ell.
\end{equation}
Moreover, if $X_{\gr{0}}$ is such that $\norm{X_{\gr{0}}}_{\B,p,q}
<\infty$, then the random field $\pr{X_{\gri_I}^I}_{\gri_I\in\Z^I}$ 
satisfies \eqref{eq:hyp_reg1} and \eqref{eq:hyp_reg2} with the 
filtration $\pr{\Fca_{\gri}}_{\gri\in\Z^d}$ replaced by 
$\pr{\Fca_{\gri_I}}_{\gri_I\in\Z^I}$. Finally, notice that 
$X_{\gri}=\sum_{I\subset\intent{1,d}}X_{\gri}^I$.
Therefore, it is possible to apply  
Theorems~\ref{thm:strong_law_large_stationary}, 
\ref{thm:strong_law_large_stationary_squares} and 
\ref{thm:Lp_law_large_stationary} to each random field 
$\pr{X_{\gri_I}^I}_{\gri_I\in\Z^I}$ in order to derive a law of large numbers.

The appropriated normalization depends on an integer $d_0$ defined as follows: for some subset $I_0$ of cardinality $d_0$, $\norm{X_{\gr{0}}^{I_0}}_{\B,1}\neq 0$ 
and for each set $I$ of cardinality strictly smaller than $d_0$, one has 
$\norm{X_{\gr{0}}^{I}}_{\B,1}=0$. We then define for 
$\grn\in\pr{\N\setminus\ens{0}}^d$, 
\begin{equation}
 \pi_{d_0,p}\pr{\grn}=\max_{I:I\subset\intent{1,d},\card{I}=d_0}
 \prod_{\ell\in I}n_\ell^{1/p}\prod_{\ell'\in\intent{1,d}\setminus I}n_{\ell'}.
\end{equation}
When all the coordinates of $\grn$ are equal, say to $N$, one has 
$\pi_{d_0,p}\pr{\grn}=\pi_{d_0,p}\pr{N\gr{1}}=N^{d_0/p+d-d_0}$.

In the previous example, if $\norm{\eps_0^{\pr{1}}}_{\R,1}\neq 0$, one has $d_0=1$ 
and if $X_{i_1,i_2}=\eps_{i_1}^{\pr{1}}\eps_{i_2}^{\pr{2}}$ then $d_0=2$. In 
higher dimension, for a prescribed $D$, one can construct similar examples 
with products over $D$ coordinates in order to get $d_0=D$.
 
It turns out that we can formulate a condition in terms of the measure of  
physical dependence introduced in \cite{MR2172215}. We define for $i\in\Z$, 
\begin{equation}
 \delta^{\pr{\ell}}_{\B,p,q}\pr{i}:=
 \norm{f_\ell\pr{\pr{\eps^{\pr{\ell},*  }_{i-u_\ell} }_{u_\ell\in\Z}   
}- f_\ell\pr{\pr{\eps^{\pr{\ell}  }_{i-u_\ell} }_{u_\ell\in\Z}   
}}_{\B,p,q},
\end{equation}
where $\eps^{\pr{\ell},*  }_{i-u_\ell}=\eps^{\pr{\ell} }_{i-u_\ell}$ 
if $u_\ell\neq i$ and $\eps^{\pr{\ell},*  }_{0}$ is a random variable 
independent of the sequence $\pr{\eps^{\pr{\ell}  }_{ u}}_{u\in\Z}$ 
and has the same distribution as $\eps^{\pr{\ell}}_0$.

We are now in position to state a strong law of large numbers for 
random fields of the form \eqref{eq:def_fonction_de_d_suites}.

\begin{Theorem}[Law of large numbers 
on rectangles]\label{thm:loi_grands_nombres_d_shifts_rectangles}
 Let $\pr{X_{\gri}}_{\gri\in\Z^d}$ be a strictly stationary random field having 
the form \eqref{eq:def_fonction_de_d_suites} and let $\alpha\in (0,1]$ 
satisfying \eqref{eq:Holder_cont_g}. Let $1/\alpha<p<r$. Suppose that
\begin{equation}\label{eq:d_shifts_hype_dep_p_d-1}
 \sum_{i\in\Z} 
\abs{i}^{d-1}\pr{\delta^{\pr{\ell}}_{\B,p\alpha,d-1}\pr{i}}^\alpha
<\infty.
\end{equation}
Then 
  \begin{equation}\label{eq:convergence_presque_sure_cas_d_shifts}
  \lim_{N\to \infty  }\sup_{\grn\smd\gr{1}, 
\max\grn \geq N}\frac 
1{\pi_{d_0,p}\pr{\grn}}\norm{\sum_{\gr{1}\imd\gri\imd\grn}X_{\gri} }  _{\B}
=0\mbox{
  almost surely.}
  \end{equation}
\end{Theorem}
\begin{Theorem}[Law of large numbers on 
squares]\label{thm:loi_grands_nombres_d_shifts_carres}
 Let $\pr{X_{\gri}}_{\gri\in\Z^d}$ be a strictly stationary random field having 
the form \eqref{eq:def_fonction_de_d_suites} and let $\alpha\in (0,1]$ 
satisfying \eqref{eq:Holder_cont_g}. Let $1/\alpha<p<r$. Suppose that for each
$\ell\in\intent{1,d}$, 
\begin{equation} 
 \sum_{i\in\Z} \abs{i}^{d-1}\pr{\delta^{\pr{\ell}}_{\B,p\alpha}\pr{i}}^\alpha 
<\infty.
\end{equation}
Then   \begin{equation}\label{eq:convergence_presque_sure_cas_fonction_d_shifts}
  \lim_{n\to \infty  } \frac 
1{n^{d_0/p+d-d_0}}\norm{\sum_{\gr{1}\imd\gri\imd n\gr{1}}X_{\gri} }  _{\B}
=0\mbox{
  almost surely}.
  \end{equation}
\end{Theorem}
\begin{Theorem}[Convergence in 
$\mathbb{L}^p$]\label{thm:loi_grands_nombres_d_shifts_conv_Lp}
 Let $\pr{X_{\gri}}_{\gri\in\Z^d}$ be a strictly stationary random field having 
the form \eqref{eq:def_fonction_de_d_suites} and let $\alpha\in (0,1]$ 
satisfying \eqref{eq:Holder_cont_g}. Let $1/\alpha<p<r$. Suppose that for each
$\ell\in\intent{1,d}$, 
\begin{equation}\label{eq:d_shifts_hype_dep_p}
 \sum_{i\in\Z} \abs{i}^{d-1}\pr{\delta^{\pr{\ell}}_{\B,p\alpha}\pr{i}}^\alpha 
<\infty.
\end{equation}
Then    
   \begin{equation}\label{eq:Lp_law_large_stationary_d_shifts}
  \lim_{N\to \infty  }\sup_{\grn\smd\gr{1}, 
\max\grn \geq N}\frac 1{\pi_{d_0,p}\pr{\grn} }\norm{\sum_{\gr{1}\imd 
  \gri\imd\grn}X_{\gri} }_{\B,p}=0.
  \end{equation}
\end{Theorem}

\section{Weak and strong law of large numbers for functions of independent 
random fields}\label{sec:lgn_Bernoulli random fields}

In this Section, we study the case where $\pr{X_{\gri}}_{\gri\in\Z^d}$ has the form 
\begin{equation}\label{eq:representation_Bernoulli_random_field}
X_{\gri}=f\pr{\pr{\eps_{\gri-\gk}}_{\gk\in\Z^d}},
\end{equation}
where $f\colon \R^{\Z^d}\to\B$ is measurable, $\pr{\B,\norm{\cdot}_{\B}}$ is a 
separable $r$-smooth Banach space and $\pr{\eps_{\gk}}_{\gk\in\Z^d}$ is an \iid 
random field.
An approach via approximation by $m$-dependent random fields can be done. In 
order to quantify this dependence, we will use the natural extension of the 
physical dependence measure to random fields, which is defined as 
\begin{equation} 
\delta_{\B,p,q}\pr{\gri}:=\norm{f\pr{\pr{\eps^*_{\gri-\gk}}_{\gk\in\Z^d}}-f\pr{
\pr {\eps_{\gri-\gk}}_{\gk\in\Z^d}}}_{\B,p, q }, \quad \gri\in\Z^d,
\end{equation}
where $\eps^*_{\gr{u}}=\eps_{\gru}$ if $\gr{u}\neq\gr{0}$
and $\eps^*_{\gr{0}}=\eps'_{\gr{0}}$, where $\eps'_{\gr{0}}$ 
is independent of $\pr{\eps_{\gru}}_{\gru\in\Z^d}$ and has the same 
distribution as $\eps_{\gr{0}}$.

For $\gri\in\Z^d$, we denote by $\norm{\gri}_{\infty}$ 
the quantity $\max_{\ell\in\intent{1,d}}\abs{i_\ell}$.
We are now in position to state a strong law of large numbers for random fields of the form \eqref{eq:representation_Bernoulli_random_field}.
\begin{Theorem}[Law of large numbers 
on rectangles]\label{thm:loi_grands_nombres_champs_bern_rect_ps}
Let $\pr{\B,\norm{\cdot}_{\B}}$ be a separable $r$-smooth Banach space and let 
$1<p< r$.
Let $\pr{X_{\gri}}_{\gri\in\Z^d}$ be a centered random field admitting the 
representation \eqref{eq:representation_Bernoulli_random_field}.
Suppose that 
 \begin{equation}
  \sum_{k=1}^\infty k^{d\pr{1-1/p}}\pr{ 
\sum_{\gri:\norm{\gri}_\infty=k}\pr{
\delta_{\B,p,d-1}\pr{\gri}}^p }^{1/p}<\infty.
 \end{equation}
 Then 
 \begin{equation}
 \lim_{N\to\infty}\sup_{\grn\smd\gr{1},\max\grn \geq N}\frac 1{\abs{\grn}^{1/p}}
\norm{ \sum_{\gr{1}\imd\gr{i}\imd \grn}X_{\gri}  
}_{\B}=0 \mbox{ a.s.}.
\end{equation}
\end{Theorem} 
 \begin{Theorem}[Law of large numbers 
on squares]\label{thm:loi_grands_nombres_champs_bern_carres}
 Let $\pr{\B,\norm{\cdot}_{\B}}$ be a separable $r$-smooth Banach space and let 
$1<p< r$.
Let $\pr{X_{\gri}}_{\gri\in\Z^d}$ be a centered random field admitting the 
representation \eqref{eq:representation_Bernoulli_random_field}.
Suppose that 
 \begin{equation}\label{eq:cond_cas_bern_carres}
  \sum_{k=1}^\infty k^{d\pr{1-1/p}}\pr{ 
\sum_{\gri:\norm{\gri}_\infty=k}\pr{\delta_{\B,p}\pr{\gri}}^p }^{1/p}<\infty.
 \end{equation}
Then 
\begin{equation}
 \lim_{n\to\infty}\frac 1{n^{d/p}}\norm{\sum_{\gr{1}\imd\gri\imd 
n\gr{1}}X_{\gri}  }_{\B}=0 \mbox{ a.s.. }
\end{equation}
\end{Theorem}
\begin{Theorem}[Convergence in 
$\mathbb{L}^p$]\label{thm:loi_grands_nombres_champs_bern_rectangles_Lp}
Let $\pr{\B,\norm{\cdot}_{\B}}$ be a separable $r$-smooth Banach space and let 
$1<p< r$.
Let $\pr{X_{\gri}}_{\gri\in\Z^d}$ be a centered random field admitting the 
representation \eqref{eq:representation_Bernoulli_random_field}.
Suppose that 
 \begin{equation} 
  \sum_{k=1}^\infty k^{d\pr{1-1/p}}\pr{ 
\sum_{\gri:\norm{\gri}_\infty=k}\pr{\delta_{\B,p}\pr{\gri}}^p }^{1/p}<\infty.
 \end{equation}
Then 
\begin{equation}\label{eq:cond_cas_bern_rectangles}
 \lim_{N\to\infty}\sup_{\grn\smd\gr{1},\max\grn \geq N}\frac 1{\abs{\grn}^{1/p}}
\norm{ \sum_{\gr{1}\imd\gr{i}\imd \grn}X_{\gri}  
}_{\B,p}=0.
\end{equation}

\end{Theorem}

We provide an application of the previous result to Hölderian functions of a 
linear random field, which are represented as 
\begin{equation}\label{eq:def_fonc_de_processus_lineaire}
 X_{\gri}=g\pr{  \sum_{\gr{k}\in\Z^d}
 A_{ \gr{k} }\pr{\eps_{\gri-\gr{k}} }},
\end{equation}
where $g\colon \B\to \B$ is $\alpha$-Hölder continuous for some $\alpha\in 
(0,1]$, that is, there exists a constant $K\pr{g}$ such that for each 
$x,x'\in\B$, 
\begin{equation}\label{eq:g_Lipschitz}
\norm{g\pr{x}-g\pr{x'}}_{\B}\leq K\pr{g}\norm{x-x'}^\alpha_{\B}, 
\end{equation}
$\pr{\B,\norm{\cdot}_{\B}}$ is a separable $r$-smooth Banach space,
$\pr{\eps_{\gru}}_{\gru\in\Z^d}$ is i.i.d.\ and $\B$-valued, $\eps_{\gr{0}}$ 
belongs to 
$\mathbb L^{p}$ for some $1<p<r$, $A_{\grk}\colon \B\to\B$ is a linear bounded 
operator, 
$\sum_{\gr{k} \in\Z^d}\norm{A_{ \grk}}_{\Bca\pr{\B}}^{p
\alpha}<\infty$, where $\norm{A_{\grk}}_{\Bca\pr{\B}} =\sup_{u\in\B,u\neq 0}
\norm{A_{\grk}\pr{u}}_{\B}/\norm{u}_{\B}$.

\begin{Corollary}\label{cor:linear_processes}
Let $\B$ be a separable $r$-smooth Banach space and $1<p<r$.
 Let $\pr{X_{\gri}}_{\gri\in\Z^d}$ be a random field defined as in
\eqref{eq:def_fonc_de_processus_lineaire}. Suppose that  $p\alpha\geq 1$ and 
that 
\begin{equation}
\sum_{k=1}^\infty k^{d\pr{1-1/p}}\pr{\sum_{\gri:\norm{\gri}_\infty =k} 
\norm{a_{\gri}  }_{\Bca\pr{\B}}^{p\alpha}
}^{1/p}<\infty. 
\end{equation}
 
\begin{itemize}
 \item If $\norm{\eps_{\gr{0}}}_{\B,p\alpha}<\infty$, then 
 \begin{equation} 
 \lim_{N\to\infty}\sup_{\grn\smd\gr{1},\max\grn \geq N}\frac 1{\abs{\grn}^{1/p}}
\norm{ \sum_{\gr{1}\imd\gr{i}\imd \grn}\pr{X_{\gri}  -\E{X_{\gri}} }
}_{\B,p}=0,
\end{equation}
and 
\begin{equation}
  \lim_{N\to\infty}\sup_{n\geq N}\frac 
1{n^{d/p}}
\norm{ \sum_{\gr{1}\imd\gr{i}\imd n\gr{1}}\pr{X_{\gri}  -\E{X_{\gri}} }
}_{\B}=0 \mbox{ a.s.}.
\end{equation}
\item If $\norm{\eps_{\gr{0}}}_{\B,p\alpha,d-1}<\infty$, then 
 \begin{equation}
 \lim_{N\to\infty}\sup_{\grn\smd\gr{1},\max\grn \geq N}\frac 1{\abs{\grn}^{1/p}}
\norm{ \sum_{\gr{1}\imd\gr{i}\imd \grn}\pr{X_{\gri}  -\E{X_{\gri}} }
}_{\B}=0 \mbox{ a.s.}.
\end{equation}
\end{itemize}
\end{Corollary}

\section{Proofs}\label{sec:proofs}

 \subsection{Proof of Theorem~\ref{thm:loi_forte_grands_nombres_carres}}

It suffices to prove that there exists a constant $C\pr{\B,d,p}$ such that for each identically distributed orthomartingale difference random field $\pr{D_{\gri}}_{\gri\in\Z^d}$ and each positive $t$, 
\begin{equation}\label{eq:goal_loi_forte_grands_nombres_carres}
\sum_{N\geq 1}\PP\pr{\max_{1\leq n\leq 2^{N+1}}\norm{S_{n\gr{1}}}_{\B}>2^{dN/p}t }\leq C\pr{\B,d,p} t^{-p}\norm{D_{\gr{0}}}_{\B,p}^p.
\end{equation}
Replacing $D_{\gri}$ by $tD_{\gri}$, it suffices to prove \eqref{eq:goal_loi_forte_grands_nombres_carres} for $t=1$. 
Define for $\ell\geq 1$ the set 
\begin{equation}
I_\ell:=\ens{\gri\in\Z^d,\gri\smd \gr{1}, \max\gri=\ell}.
\end{equation}
and 
\begin{equation}
d_\ell=\sum_{\gri\in I_\ell}D_{\gri}.
\end{equation}
For $N \geq 1$, let 
\begin{equation}
d'_{N,\ell}:= \sum_{\gri\in I_\ell}\sum_{J\subset\intent{1,d}}\pr{-1}^{\abs{J}}
 \E{D_{\gri}\ind{\norm{D_{\gri}}_{\B}\leq 2^{Nd/p}  }\mid\Fca_{\gri-\ind{J}} },
\end{equation}
\begin{equation}
d''_{N,\ell}:= \sum_{\gri\in I_\ell}\sum_{J\subset\intent{1,d}}\pr{-1}^{\abs{J}}
 \E{D_{\gri}\ind{\norm{D_{\gri}}_{\B}> 2^{Nd/p}  }\mid\Fca_{\gri-\ind{J}} }.
\end{equation}
The orthomartingale property of $\pr{D_{\gri}}_{\gri\in\Z^d}$ guarantees that 
$d_\ell=d'_{N,\ell}+d''_{N,\ell}$. Therefore, it suffices to prove that 
\begin{equation}\label{eq:conv_loi_carre_partie_bornee}
\sum_{N\geq 1}\PP\pr{\max_{1\leq n\leq 2^{N}}\norm{ 
\sum_{\ell=1}^nd'_{N,\ell}}_{\B}>2^{dN/p}}\leq C\pr{\B,d,p}
\norm{D_{\gr{0}}}_{\B,p}^p \mbox{ and}
\end{equation}
\begin{equation}\label{eq:conv_loi_carre_partie_non_bornee}
\sum_{N\geq 1}\PP\pr{\max_{1\leq n\leq 2^{N}}\norm{ 
\sum_{\ell=1}^nd''_{N,\ell}}_{\B}>2^{dN/p}}\leq C\pr{\B,d,p}
\norm{D_{\gr{0}}}_{\B,p}^p.
\end{equation}

Observe that for a fixed $N\geq 1$, the sequence $\pr{d'_{N,\ell}}_{\ell=1}^{2^N}$ is a martingale difference sequence, Doob's inequality combined with 
Proposition~\ref{prop:moments_ordre_r_orthomartingale_Banach} gives that 
\begin{equation}
\PP\pr{\max_{1\leq n\leq 2^{N}}\norm{ 
\sum_{\ell=1}^nd'_{N,\ell}}_{\B}>2^{dN/p}}
\leq 2^{-rdN/p}
\pr{\frac{r}{r-1}}^rC\pr{\B,1}\sum_{\ell=1}^{2^N}\E{\norm{d'_{N,\ell}  }_{\B}^r }.  
\end{equation}
Moreover, using the orthomartingale property of $\pr{
\sum_{J\subset\intent{1,d}}\pr{-1}^{\abs{J}}
 \E{D_{\gri}\ind{\norm{D_{\gri}}_{\B}\leq 2^{Nd/p}  }\mid\Fca_{\gri-\ind{J}} 
}}_{\gri\in\Z^d}$ followed by the triangle inequality and the fact that 
 $\pr{D_{\gri}}_{\gri\in\Z^d}$ is identically distributed, one gets 
 \begin{multline}
\PP\pr{\max_{1\leq n\leq 2^{N}}\norm{ 
\sum_{\ell=1}^nd'_{N,\ell}}_{\B}>2^{dN/p}}\\
\leq 2^{-rdN/p}
\pr{\frac{r}{r-1}}^rC\pr{\B,1}^rC\pr{\B,d}^r\sum_{\ell=1}^{2^N}
\sum_{\gri\in I_\ell} \E{\norm{\sum_{J\subset\intent{1,d}}\pr{-1}^{\abs{J}}
 \E{D_{\gri}\ind{\norm{D_{\gri}}_{\B}\leq 2^{Nd/p}  }\mid\Fca_{\gri-\ind{J}} } }_{\B}^r }\\
 \leq  2^{Nd\pr{1-r/p}}
\pr{\frac{r2^d}{r-1}}^rC\pr{\B,1}^rC\pr{\B,d}^r  \E{\norm{ 
  D_{\gr{0}}\ind{\norm{D_{\gr{0}}}_{\B}\leq 2^{Nd/p}  }  }_{\B}^r }.  
\end{multline}
Then \eqref{eq:conv_loi_carre_partie_bornee} follows from  inequality
\eqref{eq:series_queues_va_tronquee_carres}.
 
Let us show \eqref{eq:conv_loi_carre_partie_non_bornee}. Using 
$\max_{1\leq n\leq 2^{N}}\norm{ 
\sum_{\ell=1}^nd''_{N,\ell}}_{\B}
\leq \sum_{\ell=1}^{2^N}
 \norm{ d''_{N,\ell}}_{\B} $ and Markov's inequality gives 
\begin{multline*}
\PP\pr{\max_{1\leq n\leq 2^{N}}\norm{ 
\sum_{\ell=1}^nd''_{N,\ell}}_{\B}>2^{dN/p}}\leq 
2^{-dN/p}\sum_{\ell=1}^{2^N}\E{\norm{\sum_{\gri\in I_\ell}\sum_{J\subset\intent{1,d}}\pr{-1}^{\abs{J}}
 \E{D_{\gri}\ind{\norm{D_{\gri}}_{\B}> 2^{Nd/p}  }\mid\Fca_{\gri-\ind{J}} }}_{\B}}.
\end{multline*}
Then using the triangle inequality and the fact that $\pr{D_{\gri}}_{\gri\in\Z^d}$ is identically distributed gives 

\begin{equation}
\PP\pr{\max_{1\leq n\leq 2^{N}}\norm{ 
\sum_{\ell=1}^nd''_{N,\ell}}_{\B}>2^{dN/p}}\leq 
2^{d-dN\pr{1-1/p}} 2^N \E{\norm{  
 D_{\gr{0}}}_{\B}\ind{\norm{D_{\gr{0}}}_{\B}>2^{dN/p} }}.
\end{equation}
and \eqref{eq:conv_loi_carre_partie_non_bornee} follows from inequality
\eqref{eq:series_queues_va_non_tronquee_carres}. This ends the proof of 
Theorem~\ref{thm:loi_forte_grands_nombres_carres}.
\subsection{Proof of Theorem~\ref{thm:loi_faible_des_grands_nombres_orthomartingale}}

The proof will be done via a truncation argument.
Let $C\pr{\B,d,p}$ be like in Corollary~\ref{cor:moments_max_ordre_r_orthomartingale_Banach}.
Fix a positive $\eps$ and choose $R$ such that 
\begin{equation}\label{eq:choice_of_R}
 \norm{D_{\gr{0}} \ind{\norm{D_{\gr{0}}}_{\B}>R}}_{\B,p}\leq 2^{-d-1}\eps/C\pr{\B,d,p}.
\end{equation}

Define 
\begin{equation}
D'_{\gri}:=\sum_{I\subset\intent{1,d}}\pr{-1}^{\abs{I}}
\E{D_{\gri}\ind{\norm{D_{\gri}}_{\B}\leq R}\mid \mid\Fca_{\gri-\ind{I}}},
D''_{\gri}:=\sum_{I\subset\intent{1,d}}\pr{-1}^{\abs{I}}
\E{D_{\gri}\ind{\norm{D_{\gri}}_{\B}> R}\mid \mid\Fca_{\gri-\ind{I}}}.
\end{equation}
 By 
the orthomartingale property of $D_{\gri}$, one has 
$D'_{\gri}+D''_{\gri}=D_{\gri}$. Consequently, for each $\grn\smd \gr{1}$, 
\begin{equation}\label{eq:conv_Lp_orthomartingale}
 \abs{\grN}^{-1/p}\max_{\gr{1}\imd\grn\imd \grN}
\norm{S_{\grn}  }_{\B}
\leq 
 \abs{\grN}^{-1/p}\max_{\gr{1}\imd\grn\imd \grN}\norm{ \sum_{\gr{1}\imd\gri\imd\grn}D'_{\gri}  }_{\B}+
 \abs{\grN}^{-1/p}\max_{\gr{1}\imd\grn\imd \grN}\norm{ \sum_{\gr{1}\imd\gri\imd\grn}D''_{\gri}  }_{\B}.
\end{equation}
For the first term of the right hand side of \eqref{eq:conv_Lp_orthomartingale}, we use $\norm{\cdot}_{\B,p}\leq 
\norm{\cdot}_{\B,r}$ and Proposition~\ref{prop:moments_ordre_r_orthomartingale_Banach} in order to derive that 
\begin{equation}
\abs{\grN}^{-1/p}\norm{ \max_{\gr{1}\imd\grn\imd \grN}\norm{ \sum_{\gr{1}\imd\gri\imd\grn}D'_{\gri} }_{\B} }_{p}
\leq C\pr{\B,d,p}\abs{\grN}^{-1/p} 
\pr{\sum_{\gr{1}\imd \gri\imd\grN}
 \norm{D'_{\gri}}_{\B,r}^r}^{1/r}.
\end{equation}
Moreover, by definition of $D'_{\gri}$, one has 
\begin{equation}
 \norm{D'_{\gri}}_{\B,r}\leq 2^d \norm{D_{\gri}\ind{\norm{D_{\gri}}_{\B}\leq R}}_{\B,r}=2^d \norm{D_{\gr{0}}\ind{\norm{D_{\gr{0}}}_{\B}\leq R}}_{\B,r} 
\end{equation}
where we used the fact that the random field $\pr{D_{\gri}}_{\gri\in\Z^d}$ 
is identically distributed. We thus infer the bound 
\begin{equation}
\abs{\grN}^{-1/p}\norm{ \max_{\gr{1}\imd\grn\imd \grN}\norm{ \sum_{\gr{1}\imd\gri\imd\grn}D'_{\gri} }_{\B} }_{p}
\leq C\pr{\B,d,p}2^d\abs{\grN}^{1/r-1/p}
\norm{D_{\gr{0}}\ind{\norm{D_{\gr{0}}}_{\B}\leq R}}_{\B,r}.
\end{equation}
Consequently, we can find $N_0$ such that if $\max\grN\geq N_0$, then 
\begin{equation}\label{eq:norm_D'}
\abs{\grN}^{-1/p}\norm{ \max_{\gr{1}\imd\grn\imd \grN}\norm{ \sum_{\gr{1}\imd\gri\imd\grn}D'_{\gri} }_{\B} }_{p}<\eps/2.
\end{equation}
For the second term of the right hand side of \eqref{eq:conv_Lp_orthomartingale}, we apply Corollary~\ref{cor:moments_max_ordre_r_orthomartingale_Banach} and get that 
 \begin{equation}\label{eq:moment_partie_non_tronquee}
   \abs{\grN}^{-1/p}  \norm{\max_{\gr{1}\imd\grn\imd\grN}
   \norm{\sum_{\gr{1}\imd \gri\imd\grn} D''_{\gri}}_{\B} }_{p} 
  \leq C\pr{\B,d,p}\abs{\grN}^{-1/p} 
\pr{\sum_{\gr{1}\imd \gri\imd\grN}
 \norm{D''_{\gri}}_{\B,p}^p}^{1/p}.
 \end{equation}
The triangle inequality combined with the identical distribution of 
$\pr{D_{\gri}\ind{\norm{D_{\gri}}_{\B}>R}}_{\gri\in\Z^d}$ and 
\eqref{eq:choice_of_R} shows that 
\begin{equation}\label{eq:norm_D''}
\abs{\grN}^{-1/p} \norm{\max_{\gr{1}\imd\grn\imd\grN}
   \norm{\sum_{\gr{1}\imd \gri\imd\grn} D''_{\gri}}_{\B} }_{p} <\eps/2.
\end{equation}
The combination of \eqref{eq:conv_Lp_orthomartingale}, \eqref{eq:norm_D'} and \eqref{eq:norm_D''} 
concludes the proof of Theorem~\ref{thm:loi_faible_des_grands_nombres_orthomartingale}.

\subsection{Proof of Theorem~\ref{thm:strong_law_large_stationary}}

It suffices to prove that for each positive $\eps$, 
\begin{equation}
\lim_{N\to\infty}
\PP\pr{\sup_{\grn\smd\gr{1},\max\grn\geq N}\frac 1{\abs{\grn}^{1/p}}
\norm{\sum_{\gr{1}\imd\gri\imd\grn}X_{\gri}}_{\B}>\eps}=0.
\end{equation}
Let us define 
\begin{equation}\label{eq:def_de_X_i^k}
X_{\gri}^{\pr{K}}:=\sum_{\gk:  \abs{\gk-\gri}_\infty\leq K}
P_{\gri+\gk}\pr{X_{\gri}}.
\end{equation}
By assumption, the following convergence holds;
\begin{equation}
\lim_{K\to\infty}\norm{X_{\gri}-\sum_{\gk\in\Z^d:\norm{\gr{k}}_\infty \leq K}P_{\gri+\gk}\pr{X_{\gri}}}_{\B,p,d-1}=0,
\end{equation}
 where $\norm{\cdot}_{p,q}$ is defined as in \eqref{eq:def_norme_el_pq}.
By \eqref{eq:comparaison_norm_Lp_faible}, one obtains
\begin{multline}
\PP\pr{\sup_{\grn\smd\gr{1},\max\grn\geq N}\frac 1{\abs{\grn}^{1/p}}
\norm{\sum_{\gr{1}\imd\gri\imd\grn}X_{\gri}}_{\B}>\eps}\leq 
\PP\pr{\sup_{\grn\smd\gr{1},\max\grn\geq N}\frac 1{\abs{\grn}^{1/p}}
\norm{\sum_{\gr{1}\imd\gri\imd\grn}X_{\gri}^{\pr{K}}}_{\B}>\eps/2}\\
+\pr{c_p\eps/2}^{-p}  
\norm{\sup_{\grn\smd\gr{1},\max\grn\geq N}\frac 1{\abs{\grn}^{1/p}}  \norm{\sum_{\gk\in\Z^d:\norm{\gr{k}}_\infty > K}\sum_{\gr{1}\imd\gri\imd\grn}P_{\gri+\gk}\pr{X_{\gri}}}_{\B} }_{\R,p,w}^p
\end{multline}
and the triangle inequality for the norm $\norm{\cdot}_{\R,p,w}$ implies 
\begin{multline}
\PP\pr{\sup_{\grn\smd\gr{1},\max\grn\geq N}\frac 1{\abs{\grn}^{1/p}}
\norm{\sum_{\gr{1}\imd\gri\imd\grn}X_{\gri}}_{\B}>\eps}\leq 
\PP\pr{\sup_{\grn\smd\gr{1},\max\grn\geq N}\frac 1{\abs{\grn}^{1/p}}
\norm{\sum_{\gr{1}\imd\gri\imd\grn}X_{\gri}^{\pr{K}}}_{\B}>\eps/2}\\
+\pr{c_p\eps/2}^{-p}  
\pr{ \sum_{\gk\in\Z^d:\norm{\gr{k}}_\infty > K}\norm{\sup_{\grn\smd\gr{1},\max\grn\geq N}\frac 1{\abs{\grn}^{1/p}}  \sum_{\gr{1}\imd\gri\imd\grn}P_{\gri+\gk}\pr{X_{\gri}}}_{\B,p,w}  }^p.
\end{multline}
Since $\pr{P_{\gk+\gri}\pr{X_{\gri}}}_{\gri\in\Z^d}$ is an identically distributed orthomartingale difference random field, Theorem~\ref{thm:loi_forte_des_grands_nombres_orthomartingale} gives 
\begin{multline}
\PP\pr{\sup_{\grn\smd\gr{1},\max\grn\geq N}\frac 1{\abs{\grn}^{1/p}}
\norm{\sum_{\gr{1}\imd\gri\imd\grn}X_{\gri}}_{\B}>\eps}\leq 
\PP\pr{\sup_{\grn\smd\gr{1},\max\grn\geq N}\frac 1{\abs{\grn}^{1/p}}
\norm{\sum_{\gr{1}\imd\gri\imd\grn}X_{\gri}^{\pr{K}}}_{\B}>\eps/2}\\
+\pr{c_p\eps/2}^{-p}  
\pr{ \sum_{\gk\in\Z^d:\norm{\gr{k}}_\infty > K}\norm{ P_{ \gk}\pr{X_{\gr{0}}}}_{\B,p,d-1}  }^p
\end{multline}
and since we  can choose $K$ such that the last term can be made as small as we wish, it suffices to show that for each $K$, 
\begin{equation}
\lim_{N\to\infty}
\PP\pr{\sup_{\grn\smd\gr{1},\max\grn\geq N}\frac 1{\abs{\grn}^{1/p}}
\norm{\sum_{\gr{1}\imd\gri\imd\grn}X_{\gri}^{\pr{K}}}_{\B}>\eps}=0,
\end{equation}
which reduces, in view of the Borel-Cantelli lemma, to prove that
\begin{equation}\label{eq:cond_suffisante_conv_ps_orthomartingale_cobord}
\sum_{\gr{N}\in\N^d}
\PP\pr{  \frac 1{\abs{2^{\gr{N}}}^{1/p}}   \max_{\gr{1}\imd\grn\imd \gr{2^{N}}   } \norm{\sum_{\gr{1}\imd\gri\imd\grn}X_{\gri}^{\pr{K}}}_{\B}>\eps  }<\infty.
\end{equation}
By the results of Section 4 in \cite{MR3264437}, we can express $X_{\gri}^{\pr{K}}$ 
as 
\begin{equation}\label{eq:decomposition_de_X_i^k}
X_{\gri}^{\pr{K}} =
\sum_{I\subset\intent{1,d}}\prod_{\ell\in\intent{1,d}\setminus I}
\pr{\operatorname{Id}-U^{\gr{e_\ell}}}  D_{\gri}^{\pr{I}},
\end{equation}
where for $I\subsetneq\intent{1,d}$, $\pr{D_{\gri}^{\pr{I}}}_{i_\ell\in\Z,\ell\in I}$ is a strictly 
stationary orthomartingale difference random field, $D_{\gr{0}}^{\pr{I}}\in\el_{p,d-1}$. In view of 
\eqref{eq:cond_suffisante_conv_ps_orthomartingale_cobord} and 
\eqref{eq:decomposition_de_X_i^k}, it suffices to prove that for each $I\subset\intent{1,d}$, 
\begin{equation}\label{eq:cond_suffisante_conv_ps_orthomartingale_cobord_bis}
\sum_{\gr{N}\in\N^d}
\PP\pr{  \frac 1{\abs{2^{\gr{N}}}^{1/p}}   \max_{\gr{1}\imd\grn\imd \gr{2^{N}}   } \norm{\sum_{\gr{1}\imd\gri\imd\grn}U^{\gri}\prod_{\ell\in\intent{1,d}\setminus I}
\pr{\operatorname{Id}-U^{\gr{e_\ell}}}  D_{\gr{0}}^{\pr{I}}}_{\B}>\eps  }<\infty.
\end{equation}
For $I=\intent{1,d}$, this follows from Theorem~\ref{thm:loi_forte_des_grands_nombres_orthomartingale} and 
for $I=\emptyset$, one derives from 
\begin{equation}
\max_{\gr{1}\imd\grn\imd \gr{2^{N}}}\norm{\sum_{\gr{1}\imd\gri\imd\grn}U^{\gri}\prod_{\ell\in\intent{1,d} }
\pr{\operatorname{Id}-U^{\gr{e_\ell}}}  D_{\gr{0}}^{\pr{\emptyset}}}_{\B}  
\leq \max_{\gr{1}\imd\grn\imd \gr{2^{N}}+\gr{1}}
U^{\grn}\pr{\norm{ D_{\gr{0}}^{\pr{\emptyset}}}_{\B}}
\end{equation}
that 
\begin{multline}
\sum_{\gr{N}\in\N^d}
\PP\pr{  \frac 1{\abs{2^{\gr{N}}}^{1/p}}   \max_{\gr{1}\imd\grn\imd \gr{2^{N}}   } \norm{\sum_{\gr{1}\imd\gri\imd\grn}U^{\gri}\prod_{\ell\in\intent{1,d} }
\pr{\operatorname{Id}-U^{\gr{e_\ell}}}  D_{\gr{0}}^{\pr{\emptyset}}}_{\B}>\eps}  \\
\leq 2^{d}\sum_{\gr{N}\in\N^d} \abs{\gr{2^{N}}}
\PP\pr{  \norm{D_{\gr{0}}^{\pr{\emptyset}}}_{\B} >\eps\abs{\gr{2^N}}^{1/p}}
\end{multline}
and using 
\begin{equation}\label{eq:card_duples_somme_k}
\operatorname{Card}\pr{\ens{\grN\in\N^d: \sum_{\ell=1}^dN_\ell=k  }}\leq c_d k^{d-1},
\end{equation}
one finds that
\begin{multline}
\sum_{\gr{N}\in\N^d}
\PP\pr{  \frac 1{\abs{2^{\gr{N}}}^{1/p}}   \max_{\gr{1}\imd\grn\imd \gr{2^{N}}   } \norm{\sum_{\gr{1}\imd\gri\imd\grn}U^{\gri}\prod_{\ell\in\intent{1,d} }
\pr{\operatorname{Id}-U^{\gr{e_\ell}}}  D_{\gr{0}}^{\pr{\emptyset}}}_{\B}>\eps}  \\
\leq c_{d}\sum_{k=1}^\infty 2^k k^{d-1}
\PP\pr{  \norm{D_{\gr{0}}^{\pr{\emptyset}}}_{\B} >\eps 2^{k/p}},
\end{multline}
which is finite using  $D_{\gr{0}}^{\pr{\emptyset}}\in\el_{p,d-1}$
and \eqref{eq:series_queues_va_non_tronquee_rectangles}.

Since the role played by the measure preserving maps is 
symmetric, it suffices to show that for each $\ell_0\in\intent{1,d-1}$, 
\begin{equation}\label{eq:cond_suffisante_conv_ps_orthomartingale_cobord_ter}
\sum_{\gr{N}\in\N^d}
\PP\pr{  \frac 1{\abs{2^{\gr{N}}}^{1/p}}   \max_{\gr{1}\imd\grn\imd \gr{2^{N}}   } \norm{\sum_{\gr{1}\imd\gri\imd\grn}U^{\gri}\prod_{\ell\in\intent{\ell_0+1,d} }
\pr{\operatorname{Id}-U^{\gr{e_\ell}}}  D_{\gr{0}}^{\pr{\intent{1,\ell_0}}}}_{\B}>\eps  }<\infty.
\end{equation}
In order to ease the notations, we write for 
$\grn=\pr{n_\ell}_{\ell\in\intent{1,d}}$, 
$\gr{n'}\in\N^{\ell_0}:=\pr{n_\ell}_{\ell\in\intent{1,\ell_0}}$ and 
$\gr{n''}:=\pr{n_\ell}_{\ell\in\intent{\ell_0+1,d}}$, and similar notations for 
$\gr{i'},\gr{i''}$. We start from
\begin{multline}
\max_{\gr{1}\imd\grn\imd \gr{2^{N}}   } \norm{\sum_{\gr{1}\imd\gri\imd\grn}U^{\gri}\prod_{\ell\in\intent{\ell_0+1,d} }
\pr{\operatorname{Id}-U^{\gr{e_\ell}}}  
D_{\gr{0}}^{\pr{\intent{1,\ell_0}}}}_{\B}\\
\leq 
\max_{\gr{1''}\imd \gr{n''}\imd \gr{2^{N''+\gr{1''}}}}
U^{\gr{n''}}
\max_{\gr{1'}\imd\gr{n'}\imd \gr{2^{N'}}   } \norm{\sum_{\gr{1'}\imd\gri'\imd\grn'}U^{\gr{i'}}  D_{\gr{0}}^{\pr{\intent{1,\ell_0}}}}_{\B}
\end{multline}
and we are therefore reduced to show that 
\begin{equation}\label{eq:goal_decomp_ortho_cobord}
\sum_{\gr{N}\in\N^d}\abs{\gr{2^{N''}}}
\PP\pr{\max_{\gr{1'}\imd\gr{n'}\imd \gr{2^{N'}}   } \norm{\sum_{\gr{1'}\imd\gri'\imd\grn'}U^{\gr{i'}}  D_{\gr{0}}^{\pr{\intent{1,\ell_0}}}}_{\B} >\eps \abs{\gr{2^N}}^{1/p}}<\infty
\end{equation}
for each positive $\eps$. To do so, define for each $N\in\N^d$ the orthomartingale difference random fields $\pr{D_{N,\gr{0}}\circ T^{\gr{i'}}}_{\gr{i'}\in\Z^{\ell_0}}$ and $\pr{D'_{N,\gr{0}}\circ T^{\gr{i'}}}_{\gr{i'}\in\Z^{\ell_0}}$ by 
\begin{equation}
D_{N,\gr{0}}:=\sum_{J\subset\intent{1,\ell_0}}
\pr{-1}^{\card{J}}\E{D_{\gr{0}}^{\pr{\intent{1,\ell_0}}}
\ind{\norm{D_{\gr{0}}^{\pr{\intent{1,\ell_0}}}}_{\B}\leq \abs{\gr{2^{N}}}^{1/p}} \mid \Fca_{\gr{0}-\ind{J}}},
\end{equation}
\begin{equation}
D'_{N,\gr{0}}:=\sum_{J\subset\intent{1,\ell_0}}
\pr{-1}^{\card{J}}\E{D_{\gr{0}}^{\pr{\intent{1,\ell_0}}}
\ind{\norm{D_{\gr{0}}^{\pr{\intent{1,\ell_0}}}}_{\B}> \abs{\gr{2^{N}}}^{1/p}} \mid \Fca_{\gr{0}-\ind{J}}}. 
\end{equation}
The equality $D_{\gr{0}}^{\pr{\intent{1,\ell_0}}}=D_{N,\gr{0}}+
D'_{N,\gr{0}}$ reduces the proof of \eqref{eq:goal_decomp_ortho_cobord} 
to 
\begin{equation}\label{eq:goal_decomp_ortho_cobord1}
\sum_{\gr{N}\in\N^d}\abs{\gr{2^{N''}}}
\PP\pr{\max_{\gr{1'}\imd\gr{n'}\imd \gr{2^{N'}}   } \norm{\sum_{\gr{1'}\imd\gri'\imd\grn'}U^{\gr{i'}} D_{N,\gr{0}} }_{\B} >\eps \abs{\gr{2^N}}^{1/p}}<\infty\mbox{ and }
\end{equation}
\begin{equation}\label{eq:goal_decomp_ortho_cobord2}
\sum_{\gr{N}\in\N^d}\abs{\gr{2^{N''}}}
\PP\pr{\max_{\gr{1'}\imd\gr{n'}\imd \gr{2^{N'}}   } \norm{\sum_{\gr{1'}\imd\gri'\imd\grn'}U^{\gr{i'}} D'_{N,\gr{0}}}_{\B} >\eps \abs{\gr{2^N}}^{1/p}}<\infty.
\end{equation}
In order to show \eqref{eq:goal_decomp_ortho_cobord1}, we 
use Corollary~\ref{cor:moments_max_ordre_r_orthomartingale_Banach} with $p=r$ combined with the identical distribution of 
$\pr{U^{\gr{i'}}D_{N,\gr{0}}}_{\gr{i'}\in\Z^{\ell_0}}$ and derive that 
\begin{multline}
\abs{\gr{2^{N''}}}
\PP\pr{\max_{\gr{1'}\imd\gr{n'}\imd \gr{2^{N'}}   } \norm{\sum_{\gr{1'}\imd\gri'\imd\grn'}U^{\gr{i'}} D_{N,\gr{0}} }_{\B} >\eps \abs{\gr{2^N}}^{1/p}}\\
\leq \eps^{-r}C\pr{\B,d,r}\abs{\gr{2^N}}^{-r/p}\abs{2^{\gr{N}}}\norm{D_{N,\gr{0}}}_{\B,r}^r\\
\leq   \eps^{-r}C\pr{\B,d,r}2^{\ell_0}\abs{\gr{2^N}}^{\pr{1-r/p}} 
\norm{D_{\gr{0}}^{\pr{\intent{1,\ell_0}}}
\ind{\norm{D_{\gr{0}}^{\pr{\intent{1,\ell_0}}}}_{\B}\leq \abs{\gr{2^{N}}}^{1/p}}}_{\B,r}^r.
\end{multline}
Then by \eqref{eq:card_duples_somme_k}, we reduce the proof of 
\eqref{eq:goal_decomp_ortho_cobord1} to 
\begin{equation}\label{eq:tail_series_for_Hannan_rect1}
\sum_{k=1}^\infty 2^{k\pr{1-r/p}}k^{d-1}\norm{D_{\gr{0}}^{\pr{\intent{1,\ell_0}}}
\ind{\norm{D_{\gr{0}}^{\pr{\intent{1,\ell_0}}}}_{\B}\leq 
2^{k/p}}}_{\B,r}^r<\infty,
\end{equation}
which follows from \eqref{eq:series_queues_va_tronquee_rectangles}. In 
order to show \eqref{eq:goal_decomp_ortho_cobord2}, we start by Markov's inequality, which gives 
\begin{equation*}
\abs{\gr{2^{N''}}}
\PP\pr{\max_{\gr{1'}\imd\gr{n'}\imd \gr{2^{N'}}   } \norm{\sum_{\gr{1'}\imd\gri'\imd\grn'}U^{\gr{i'}} D'_{N,\gr{0}}}_{\B} >\eps \abs{\gr{2^N}}}
\leq 2^d\eps^{-1}\abs{\gr{2^{N}}}
\E{\norm{D_{\gr{0}}^{\pr{\intent{1,\ell_0}}}}_{\B}
\ind{\norm{D_{\gr{0}}^{\pr{\intent{1,\ell_0}}}}_{\B}>\abs{\gr{2^{N}}}^{1/p} }  },
\end{equation*}
then we sum over $\gr{N}\in\N^d$, use \eqref{eq:card_duples_somme_k} 
in order to reduce the proof of \eqref{eq:goal_decomp_ortho_cobord2}
to that of 
\begin{equation}\label{eq:tail_series_for_Hannan_rect2}
 \sum_{k=0}^\infty 
2^kk^{d-1}\E{\norm{D_{\gr{0}}^{\pr{\intent{1,\ell_0}}}}_{\B}\ind{\norm{D_{\gr{0}
}^{\pr{\intent{1,\ell_0}}}}_{\B}>2^{k/p}}}<\infty.
\end{equation}
Then we use \eqref{eq:series_E_Y_ind_Y>} to get  
\eqref{eq:goal_decomp_ortho_cobord2} and conclude the proof of 
Theorem~\ref{thm:strong_law_large_stationary}.

\subsection{Proof of Theorem~\ref{thm:strong_law_large_stationary_squares}}

The proof follows essentially the lines of that of 
Theorem~\ref{thm:strong_law_large_stationary} hence 
we will mention only the crucial steps. The reduction to the case where the 
random field $\pr{X_{\gri}}_{\gri\in\Z^d}$ admits the representation 
\eqref{eq:decomposition_de_X_i^k} works in an analoguous way, replacing each 
occurence of $\sup_{\grn\smd\gr{1},\max\grn\geq 
N}$ by a supremum over the $\grn$ having identical coordinates. 
The analogue of \eqref{eq:cond_suffisante_conv_ps_orthomartingale_cobord_bis} 
reads as 
\begin{equation}
\label{eq:cond_suffisante_conv_ps_orthomartingale_cobord__carres_bis}
\sum_{N=1}^\infty
\PP\pr{  \frac 1{2^{Nd/p}}  \max_{1\leq n\leq 2^N} 
\norm{\sum_{\gr{1}\imd\gri\imd 
n\gr{1}}U^{\gri}\prod_{\ell\in\intent{1,d}\setminus I}
\pr{\operatorname{Id}-U^{\gr{e_\ell}}}  D_{\gr{0}}^{\pr{I}}}_{\B}>\eps  
}<\infty.
\end{equation}
For $I=\intent{1,d}$, this follows from 
Theorem~\ref{thm:loi_forte_grands_nombres_carres} and for $I=\emptyset$, we use 
the bound 
\begin{equation}
 \norm{\sum_{\gr{1}\imd\gri\imd 
n\gr{1}}U^{\gri}\prod_{\ell\in\intent{1,d} }
\pr{\operatorname{Id}-U^{\gr{e_\ell}}}  D_{\gr{0}}^{\pr{\emptyset}}}_{\B} \leq 
U^{\gr{1}}\max_{J\subset\intent{1,d}}U^{n\gr{1_J}}\pr{ 
\norm{D_{\gr{0}}^{\pr{\emptyset}}}_{\B}}
\end{equation}
from which it follows that
\begin{equation}
 \PP\pr{  \frac 1{2^{Nd/p}}  \max_{1\leq n\leq 2^N} 
\norm{\sum_{\gr{1}\imd\gri\imd 
n\gr{1}}U^{\gri}\prod_{\ell\in\intent{1,d}\setminus I}
\pr{\operatorname{Id}-U^{\gr{e_\ell}}}  D_{\gr{0}}^{\pr{I}}}_{\B}>\eps  
} \leq 2^{Nd}\PP\pr{\norm{D_{\gr{0}}^{\pr{\emptyset}}}_{\B}>\eps 2^{Nd/p}   }
\end{equation}
and the convergence of the series in 
\eqref{eq:cond_suffisante_conv_ps_orthomartingale_cobord__carres_bis} is a 
consequence of the fact that $\norm{D_{\gr{0}}^{\pr{\emptyset}}}_{\B}$ belongs 
to $\mathbb L^p$. For $\emptyset\subsetneq I\subsetneq \intent{1,d}$, 
we use the same symmetry argument to deal with the case $I=\intent{1,\ell_0}$ 
for some $\ell_0\in\intent{1,d-1}$. Denoting as before 
$\gr{i'}=\pr{i_\ell}_{\ell\in\intent{1,\ell_0}}$ for 
$\gri=\pr{i_\ell}_{\ell\in\intent{1,d}}$, we have 
\begin{equation}
 \max_{1\leq n\leq 2^N} 
\norm{\sum_{\gr{1}\imd\gri\imd 
n\gr{1}}U^{\gri}\prod_{\ell\in\intent{\ell_0+1,d}}
\pr{\operatorname{Id}-U^{\gr{e_\ell}}}  D_{\gr{0}}^{\pr{I}}}_{\B}
\leq U^{\gr{1''}}\max_{J\subset\intent{\ell_0+1,d}}
\max_{1\leq n\leq 2^N}U^{n\gr{1_J}}\norm{\sum_{\gr{1'}\imd \gr{i'}\imd n\gr{1'}  
} 
D_{\gr{0}}^{\pr{\intent{1,\ell_0}}}  }_{\B}
\end{equation}
hence 
\begin{multline}
 \PP\pr{ \max_{1\leq n\leq 2^N} 
\norm{\sum_{\gr{1}\imd\gri\imd 
n\gr{1}}U^{\gri}\prod_{\ell\in\intent{1,d}\setminus I}
\pr{\operatorname{Id}-U^{\gr{e_\ell}}}  D_{\gr{0}}^{\pr{I}}}_{\B}>\eps 
2^{Nd/p}}\\
\leq 2^{d+\ell_0 N}\PP\pr{ \max_{1\leq n\leq 
2^N}U^{n\gr{1_J}}\norm{\sum_{\gr{1'}\imd \gr{i'}\imd n\gr{1'}  
} 
D_{\gr{0}}^{\pr{\intent{1,\ell_0}}}  }_{\B}>\eps 2^{Nd/p}   }.
\end{multline}
The control of the tail is done in a similar way, this time with $2^{Nd/p}$ as 
a truncation level, which leads to show the convergence of analogous series as 
in \eqref{eq:tail_series_for_Hannan_rect1} and 
\eqref{eq:tail_series_for_Hannan_rect2} but without the term $k^{d-1}$
and where $k$ is a multiple of $d$, which follows from the fact that 
$\norm{D_{\gr{0}}^{\pr{\intent{1,\ell_0}}}  }_{\B}$ belongs to $\mathbb L^p$. 
This ends the proof Theorem~\ref{thm:strong_law_large_stationary_squares}.

\subsection{Proof of Theorem~\ref{thm:Lp_law_large_stationary}}

Define $X_{\gri}^{\pr{K}}$ as in \eqref{eq:def_de_X_i^k}.
Observe that $X_{\gri}=X_{\gri}^{\pr{K}}+\sum_{\gk:  \abs{\gk-\gri}_\infty> K}
P_{\gk}\pr{X_{\gri}}$, hence 
\begin{multline}
\frac 1{\abs{\grN}^{1/p}}\norm{\max_{\gr{1}\imd\grn\imd\grN}
\norm{\sum_{\gr{1}\imd 
  \gri\imd\grn}X_{\gri}}_{\B} }_{p}\\ \leq 
  \frac 1{\abs{\grN}^{1/p}}\norm{\max_{\gr{1}\imd\grn\imd\grN}\norm{\sum_{\gr{1}\imd 
  \gri\imd\grn}X_{\gri}^{\pr{K}}}_{\B} }_{p}+
  \sum_{\gk:  \abs{\gk-\gri}_\infty> K}\frac 1{\abs{\grN}^{1/p}}\norm{\max_{\gr{1}\imd\grn\imd\grN}\norm{\sum_{\gr{1}\imd 
  \gri\imd\grn}P_{\gri+\gk}\pr{X_{\gri}}}_{\B} }_{p}.
\end{multline}
Moreover, using Corollary~\ref{cor:moments_max_ordre_r_orthomartingale_Banach} 
then \eqref{eq:stationarity_projecteurs}, one has 
\begin{align}
\frac 1{\abs{\grN}^{1/p}}\norm{\max_{\gr{1}\imd\grn\imd\grN}\norm{\sum_{\gr{1}\imd 
  \gri\imd\grn}P_{\gri+\gk}\pr{X_{\gri}}}_{\B} }_{p}
  &\leq \frac{C\pr{\B,d,p} }{\abs{\grN}^{1/p}}
  \pr{\sum_{\gr{1}\imd 
  \gri\imd\grN}\norm{P_{\gri+\gk}\pr{X_{\gri}} }_{\B,p}^p}^{1/p}\\
  &\leq C\pr{\B,d,p}\norm{P_{ \gk}\pr{X_{\gr{0}}} }_{\B,p}
\end{align}
hence in view of assumption \eqref{eq:Hannan_el_p}, it suffices to prove that 
for each $K$, 
\begin{equation}
   \lim_{R\to \infty  }\sup_{\grN\smd\gr{1}, 
\max\grN \geq R}\frac 1{\abs{\grN}^{1/p}}\norm{\max_{\gr{1}\imd\grn\imd\grN}\norm{\sum_{\gr{1}\imd 
  \gri\imd\grn}U^{\gri}X_{\gr{0}}^{\pr{K}}}_{\B} }_{p}=0.
\end{equation}
 In view of the decomposition
\eqref{eq:decomposition_de_X_i^k}, this reduces to check that for 
each $I\subset\intent{1,d}$, 
\begin{equation}\label{eq:dec_orthom_cob_loi_Lp_ens_I}
   \lim_{R\to \infty  }\sup_{\grN\smd\gr{1}, 
\max\grN \geq R}\frac 1{\abs{\grN}^{1/p}}\norm{ \max_{\gr{1}\imd\grn\imd\grN}\norm{\sum_{\gr{1}\imd 
  \gri\imd\grn}\prod_{\ell\in\intent{1,d}\setminus I}
\pr{\Id-U^{\gr{e_\ell}}}U^{\gri} D_{\gr{0}}^{\pr{I}}}_{\B}  }_{p}=0.
\end{equation}
The case $I=\intent{1,d}$ corresponds to Theorem~\ref{thm:loi_faible_des_grands_nombres_orthomartingale}.
Moreover, for $I=\emptyset$, using $\norm{\max_{j\in J}\norm{Y_j}_{\B}  }_p\leq \card{J}^{1/p}\max_{j\in J}\norm{Y_j}_{\B,p}$,  one has 
\begin{align}
\norm{\max_{\gr{1}\imd\gri\imd\grN}\norm{\sum_{\gr{1}\imd 
  \gri\imd\grn}\prod_{\ell\in\intent{1,d} }
\pr{\Id-U^{\gr{e_\ell}}}  D_{\gri}^{\pr{\emptyset}}}_{\B}  }_{p}&\leq 2^d
\norm{\max_{\gr{1}\imd\gri\imd\grN+\gr{1}}\norm{ 
U^{\gri}D_{\gr{0}}^{\pr{\emptyset}}
 }_{\B}}_p\\
 &\leq 2^d\tau+\abs{\grN+\gr{1}}^{1/p}\norm{D_{\gr{0}}^{\pr{\emptyset}}
 \ind{\norm{D_{\gr{0}}^{\pr{\emptyset}}}_{\B}>\tau}}_p
\end{align}
hence \eqref{eq:dec_orthom_cob_loi_Lp_ens_I} holds for $I=\emptyset$.
Since the role played by the coordinates is symmetric, it suffices to 
treat the case $I=\intent{1,\ell_0}$ for each $\ell_0\in\intent{1,d-1}$. 
For such an $\ell_0$, denote for $\gri=\pr{i_\ell}_{\ell\in\intent{1,d}}\in\N^d$, $\gr{i'}:=\pr{i_\ell}_{\ell\in\intent{1,\ell_0}}$ and 
$\gr{i''}:=\pr{i_\ell}_{\ell\in\intent{\ell_0+1,d}}$ and similarly for $\gr{n'}, 
\gr{n''}$, $\gr{1'}$ and $\gr{1''}$. One has 
\begin{equation}
\max_{\gr{1}\imd\grn\imd\grN}\norm{\sum_{\gr{1}\imd 
  \gri\imd\grn}\prod_{\ell\in\intent{1,d}\setminus I}
\pr{\Id-U^{\gr{e_\ell}}} U^{\gri} D_{\gr{0}}^{\pr{I}}}_{\B}
\leq \max_{\gr{1''}\imd\gr{i''}\imd\gr{N''}+\gr{1}}
U^{\gr{i''}}
\max_{\gr{1'}\imd\gr{n'}\imd \gr{N'}}\norm{
\sum_{\gr{1'}\imd \gr{i'}\imd\gr{n'}}
U^{\gr{i'}}D_{\gr{0}}^{\pr{I}}}_{\B}.
\end{equation}
and by Theorem~\ref{thm:loi_faible_des_grands_nombres_orthomartingale}, the 
family 
\begin{equation*}
\ens{M_{\gr{N'}}:=\frac 1{\abs{\gr{N'}} }\max_{\gr{1'}\imd\gr{n'}\imd \gr{N'}}\norm{
\sum_{\gr{1'}\imd \gr{i'}\imd\gr{n'}}
U^{\gr{i'}}D_{\gr{0}}^{\pr{I}}}_{\B}^p,\gr{N'}\in\N^{\ell_0}},
\end{equation*}
 is uniformly integrable and $\lim_{R\to\infty}\max_{\gr{N'}:\max\gr{N'}\geq R}\norm{M_{\gr{N'}}}_{\B,p}=0$. With the observation that 
 \begin{align}
 A_{\grN}&:=\frac 
1{\abs{\grN}^{1/p}}\norm{\max_{\gr{1}\imd\grn\imd\grN}\norm{\sum_{\gr{1}\imd 
  \gri\imd\grn}\prod_{\ell\in\intent{1,d}\setminus I}
\pr{\Id-U^{\gr{e_\ell}}} U^{\gri} D_{\gr{0}}^{\pr{I}}}_{\B}}_p \\
&\leq \frac 
1{\abs{\gr{N''}}^{1/p}}\max_{\gr{1''}\imd\gr{i''}\imd\gr{N''}+\gr{1}}
U^{\gr{i''}}M^{1/p}_{\gr{N'}},
 \end{align}
 one get 
 \begin{equation}
 A_{\grN}\leq  2^p\min\ens{\norm{M_{\gr{N'}}^{1/p}}_{\B,p}, 
 \frac{\tau}{\abs{\gr{N''}}^{1/p}   }+ \norm{M_{\gr{N'}}\ind{\norm{M_{\gr{N'}}}_{\B}>\tau}  }_{\B,p}  }.
 \end{equation}
This ends the proof of Theorem~\ref{thm:loi_faible_des_grands_nombres_orthomartingale}.

\subsection{Proof of the results of Section~\ref{sec:LLN_shifts}}

As pointed out before the statements of theorems, we can decompose the
random field $\pr{X_{\gri}}_{\gri\in\Z^d}$ as a sum indexed by subsets $I$ of $\intent{1,d}$ of random fields whose coordinates are the restriction of those of $\gri$ to the set $I$, namely,
\begin{equation}
 X_{\gri}=\sum_{I\subset\intent{1,d},\card{I}\geq d_0}X_{\gr{i_I}}^I,
\end{equation}
where $X_{\gr{i_I}}^I$ is defined as in \eqref{eq:def_X_gri_I_d_shifts}.
Consequently, for each $\grn\smd\gr{1}$,
\begin{equation}
 \frac 1{\pi_{d_0,p}\pr{\grn}}\norm{\sum_{\gr{1}\imd\gri\imd\grn}X_{\gri}}_{\B}
 \leq \sum_{I\subset\intent{1,d},\card{I}\geq d_0}
 \frac 1{\prod_{\ell\in I}n_i^{1/p}}
 \norm{\sum_{\gr{1_I}\imd\gr{i_I}\imd\gr{n_I}}X_{\gr{1_I}}^I}_{\B}.
\end{equation}
 We are thus reduced to show that conditions \eqref{eq:d_shifts_hype_dep_p} and \eqref{eq:d_shifts_hype_dep_p_d-1} imply that for each $I\subset\intent{1,d}$ having $d_0$ or more elements, the random field $\pr{X_{\gr{i_I}}^I }_{\gr{i_I}\in\Z^I}$
 satisfies the assumptions of Theorems~\ref{thm:strong_law_large_stationary},  \ref{thm:strong_law_large_stationary_squares} and
 \ref{thm:Lp_law_large_stationary}. By the symmetric role played by the coordinates, it suffices to show that
\eqref{eq:d_shifts_hype_dep_p} (respectively  \eqref{eq:d_shifts_hype_dep_p_d-1}) implies that
 for each $\ell_0\in\intent{d_0,d}$,
\begin{equation}
\sum_{\gk\in\Z^{\ell_0}}\norm{P_{\gk}\pr{X^{\intent{1,\ell_0}}_{\gr{0}}}}_{\B,p}<\infty
\end{equation}
(respectively
\begin{equation}
\sum_{\gk\in\Z^{\ell_0}}\norm{P_{\gk}\pr{X^{\intent{1,\ell_0}}_{\gr{0}}}}_{\B,p,d-1}<\infty).
\end{equation}
To do so, we shall first prove that for $q\geq 0$,
\begin{equation}\label{eq:cond_norm_Lpq_coeff_d_shifts}
\forall\ell\in\intent{1,d}, \sum_{i\in\Z} 
\abs{i}^{d-1}\pr{\delta^{\pr{\ell}}_{\B,p\alpha,q}\pr{i}}^\alpha
<\infty
\end{equation}
implies that
\begin{equation}\label{eq:Hannan_el_pq_dshifts}
\sum_{\gk\in\Z^d}\norm{P_{\gk}\pr{X_{\gr{0}}}}_{\B,p,q}<\infty,
\end{equation}
then that \eqref{eq:Hannan_el_pq_dshifts} implies that for each $\ell_0\in\intent{d_0,d}$,
\begin{equation}\label{eq:Hannan_el_pq_dshifts_pour_XiI}
\sum_{\gk\in\Z^{\ell_0}}\norm{P_{\gk}\pr{X^{\intent{1,\ell_0}}_{\gr{0}}}}_{\B,p,q}<\infty.
\end{equation}
\begin{itemize}
  \item Proof that \eqref{eq:cond_norm_Lpq_coeff_d_shifts}
  implies \eqref{eq:Hannan_el_pq_dshifts}.

  First, observe that by commutativity of the operators $P_{k
    \gr{e}_\ell}$ and the fact that $\norm{P_{\grk-k_\ell\gr{e}_\ell}\pr{X_0}}_{\B,p,q}\leq 2^{d-1}\norm{X_{\gr{0}}}_{\B,p,q}$, we get that for each $\grk\in\Z^d$,
  \begin{equation}
    \norm{P_{\gk}\pr{X_{\gr{0}}}}_{\B,p,q}\leq \min_{1\leq\ell\leq d}
    \norm{P_{k
        \gr{e}_\ell}\pr{X_{\gr{0}}}}_{\B,p,q}.
  \end{equation}
  Define the random variable $X_{\gr{0}}^{\pr{\grk,\ell}}$ by
  \begin{equation}
    X_{\gr{0}}^{\pr{\grk,\ell}}=g\pr{ f_1\pr{ \pr{\eps^{*,\pr{1}}_{i_1-u_1}}_{u_1\in\Z} },\dots,
f_d\pr{ \pr{\eps^{*,\pr{d}}_{i_d-u_d}}_{u_d\in\Z} }   },
  \end{equation}
  where $\eps^{*,\pr{\ell'}}_{u}=\eps^{\pr{\ell'}}$ if $\ell\neq \ell'$, $\eps^{*,\pr{\ell}}_u=\eps_u$ for $u\neq k_\ell$ and
  $\eps^{*,\pr{\ell}}_{k_\ell}=\eps^{',\pr{\ell}}_{k_\ell}$. In other words, only the random variable $\eps^{\pr{\ell}}_{k_\ell}$ is replaced by a copy independent of the sequences $\pr{\eps^{\pr{q}}_{u_q}}_{u_q\in\Z^d}$.
  Since $P_{k_\ell\gr{e}_\ell}\pr{X_{\gr{0}}^{\pr{\grk,\ell}}}=0$, we get that
    \begin{equation}
    \norm{P_{\gk}\pr{X_{\gr{0}}}}_{\B,p,q}\leq 2\min_{1\leq\ell\leq d}
    \norm{X_{\gr{0}}- X_{\gr{0}}^{\pr{\grk,\ell}}}_{\B,p,q}.
  \end{equation}.
  Using Hölder regularity of $g$ and inequality \eqref{eq:norme_el_pq_de_puissances}, we derive that
  \begin{align}
   \norm{P_{\gk}\pr{X_{\gr{0}}}}_{\B,p,q}
   &\leq 2C \min_{1\leq\ell\leq d}
   \norm{\abs{f_\ell\pr{\pr{\eps^{\pr{\ell},*  }_{k_\ell-u_\ell} }_{u_\ell\in\Z}
}- f_\ell\pr{\pr{\eps^{\pr{\ell}  }_{k_\ell-u_\ell} }_{u_\ell\in\Z}
}}^{\alpha} }_{\B,p,q}\\
&\leq 2C\kappa_{\alpha,p,q} \min_{1\leq\ell\leq d} \norm{ f_\ell\pr{\pr{\eps^{\pr{\ell},*  }_{k_\ell-u_\ell} }_{u_\ell\in\Z}
}- f_\ell\pr{\pr{\eps^{\pr{\ell}  }_{k_\ell-u_\ell} }_{u_\ell\in\Z}
}  }_{\B,p\alpha,q}^\alpha\\
&=2C\kappa_{\alpha,p,q} \min_{1\leq\ell\leq d}
\pr{\delta^{\pr{\ell}}_{\B,p\alpha,q}\pr{k_\ell}}^\alpha.
  \end{align}
  Defining $a_k:=\max_{1\leq\ell\leq d}\pr{\delta^{\pr{\ell}}_{\B,p\alpha,q}\pr{k}}^\alpha$, we are thus reduced to prove that
  $\sum_{\grk\in\Z^d}\min_{1\leq\ell\leq d}a_{k_\ell}<\infty$.
  To do so, take a bijective map $\tau\colon \N\to\Z$ such that the sequence 
$\pr{b_i}_{i\in\N}=\pr{a_{\tau\pr{i}}}_{i\in\N}$ is non-increasing. We have to 
show that $\sum_{i_1,\dots,i_d\in\N}c_{i_1,\dots,i_d}<\infty$, where 
$c_{i_1,\dots,i_d}=\min_{1\leq \ell\leq d}b_{i_\ell}$. By invariance of
$c_{i_1,\dots,i_d}$ under permutation of the indexes $i_1,\dots,i_d$, it 
suffices to prove that $\sum_{i_1,\dots,i_d,  0\leq i_1, \dots, i_{d-1}\leq 
i_d}c_{i_1,\dots,i_d}<\infty$. Since $\pr{b_i}_{i\geq 1}$ is non-increasing, we 
derive that for each $i_d\geq 0$, $\sum_{0\leq i_1,\dots,i_{d-1}\leq i_d}
\leq b_{i_d} i_d^{d-1}$. Summing over $i_d$ and using 
\eqref{eq:cond_norm_Lpq_coeff_d_shifts} allows us to derive 
\eqref{eq:Hannan_el_pq_dshifts}.

  \item Proof that \eqref{eq:Hannan_el_pq_dshifts} implies 
\eqref{eq:Hannan_el_pq_dshifts_pour_XiI}
   
   From the definition of $X_{\gr{0}}^{\intent{1,\ell_0}}$ given 
in \eqref{eq:def_X_gri_I_d_shifts}, the following equality holds for each 
$\grk\in\Z^{\ell_0}$:
\begin{equation}\label{eq:projecteurs_X0I}
 P_{\grk}\pr{X_{\gr{0}}^{\intent{1,\ell_0}}  
}=\sum_{K\subset\intent{1,\ell_0} 
}\pr{-1}^{\card{K}}\sum_{J\subset\intent{1,\ell_0} }\pr{-1}^{\card{J}}
\E{\E{X_{\gr{0}} \mid\Gca_J } \mid\Fca_{\grk-\ind{K}}  }.
\end{equation}
Observe that if $J$ is such $\intent{1,\ell_0}\setminus J$ contains some $j_0$, 
then for each $K\subset\intent{1,\ell_0}$ such that $j_0\notin K$,
\begin{equation}
 \E{\E{X_{\gr{0}} \mid\Gca_J } \mid\Fca_{\grk-\ind{K}}  }
 =\E{\E{X_{\gr{0}} \mid\Gca_J } \mid\Fca_{\grk-\ind{K \cup\ens{j_0}}}  }.
\end{equation}
As a consequence, only the term for $J=\intent{1,\ell_0}$ in the right hand 
side of \eqref{eq:projecteurs_X0I} remains  and using commutativity of 
$\pr{\Fca_{\gri}}_{\gri\in\Z^d}$ gives 
\begin{equation}
 P_{\grk}\pr{X_{\gr{0}}^{\intent{1,\ell_0}}  
}=\pr{-1}^{\ell_0}\sum_{K\subset\intent{1,\ell_0} 
}\pr{-1}^{\card{K}} 
\E{\E{X_{\gr{0}} \mid\mid\Fca_{\grk-\ind{K}}  } \mid\Gca_{\intent{1,\ell_0}}  }
\end{equation}
and it follows that 
\begin{equation}
 \norm{ P_{\grk}\pr{X_{\gr{0}}^{\intent{1,\ell_0}}  }}_{\B,p,q}\leq 
 \norm{P_{k_1,\dots,k_{\ell_0},0,\dots,0  }\pr{X_{\gr{0}}}  }_{\B,p,q},
\end{equation}
from which \eqref{eq:Hannan_el_pq_dshifts_pour_XiI} can be easily derived.
\end{itemize}
This ends the proof of the results of Section~\ref{sec:LLN_shifts}.

\subsection{Proof of the results of Section~\ref{sec:lgn_Bernoulli random 
fields}}

The proof of the results of Section~\ref{sec:lgn_Bernoulli random 
fields}  will first require some 
preliminary notations and intermediate results.
Define for $m\geq 0$ and $\gri\in\Z^d$ the random variables
\begin{equation}
 Y_{\gri,m}=\E{X_{\gri}\mid\sigma\pr{\eps_{\grk}:
\norm{\grk-\gri}_\infty\leq m}},\quad  X_{\gri,m}=Y_{\gri,m}-Y_{\gri,m-1}
\end{equation}
 and $Y_{\gri,-1}=0$.
Then $X_{\gri}=\sum_{m=0}^\infty X_{\gri,m}$ and the convergence holds in the 
almost sure sense. 
For a fixed $m$, define $m':=2m+1$ and for $\gr{a}\in\Z^d$ such that 
$\gr{1}\imd\gr{a}\imd m'\gr{1}$, let $I_{\gr{a}}$ be set of elements 
$\gri\in\Z^d$ such that there exists $\grj\in\Z^d$ for which 
$\gri=m'\grj+\gr{a}$. Notice that $\pr{X_{\gri,m}}_{\gri\in I_{\gr{a}}}$ is an 
i.i.d.\ random field, or equivalently, that 
$\pr{X_{m'\grj+\gr{a}},m}_{\grj\in\Z^d}$ is an i.i.d.\ random field. 

Also, observe that for each integer $m\geq 0$ and each $\grn\smd\gr{1}$, it is 
possible to bound the partial sums of $\pr{X_{\gri}}_{\gri\in\Z^d}$ over 
rectangles via partial sums of the i.i.d.\ random fields 
$\pr{X_{m'\grj+\gr{a}},m}_{\grj\in\Z^d}$, namely, 
\begin{equation}\label{eq:decomposition_champ_m_dependent}
 \norm{\sum_{\gr{1}\imd\gri\imd \grn}X_{\gri} }_{\B}
 \leq \sum_{\gr{1}\imd\gr{a}\imd m'\gr{1}}
 \sum_{\gr{\delta}\in\ens{0,1}^d}
 \norm{\sum_{\gr{0}\imd\grj \imd \ent{\frac{1}{m'}\grn     } -\gr{\delta} }
 X_{m'\grj+\gr{a},m}}_{\B}, 
\end{equation}
where for $\gr{x}=\pr{x_\ell}_{\ell\in\intent{1,d}}\in\R^d$, $\ent{\gr{x}}
=\pr{\ent{x_\ell}}_{\ell\in\intent{1,d}}$ and for a real number $t$, $\ent{t}$ 
is the unique integer satisfying $\ent{t}\leq t<\ent{t}+1$.
Indeed, since the sets $I_{\gr{a}}$, $\gr{1}\imd\gr{a}\imd m'\gr{1}$ are 
pairwise disjoint, the following inequality takes place 
\begin{equation}
 \norm{\sum_{\gr{1}\imd\gri\imd \grn}X_{\gri} }_{\B}
 \leq \sum_{\gr{1}\imd\gr{a}\imd m'\gr{1}}
 \norm{\sum_{\gr{1}\imd\gri\imd \grn,\gri\in I_{\gr{a}}}X_{\gri} }_{\B}.
\end{equation}
Then we express $\gri\in I_{\gr{a}}$ as $\gri=m'\grj+\gr{a}$ and translate 
the inequalities $\gr{1}\imd\gri\imd \grn$ as  
$\gr{0}\imd\grj\imd\ent{m^{-1}\pr{\grn -\gr{a}}}$. The sum over 
$\gr{\delta}\in\ens{0,1}^d$ comes from the fact that for each $\ell$, 
$\ent{m^{-1}\pr{n_\ell 
-a_\ell}}\in\ens{\ent{m^{-1}n_\ell},\ent{m^{-1}n_\ell}-1}$.

We will express bounds on norm of maximal functions in terms of 
some norm of $X_{\gr{0},m}$. For this reason, we need the following intermediate 
result.

\begin{Lemma}\label{lem:norme_Orlicz_X_0m}
Let $1<p\leq r$ and $q\geq 0$. There exists a constant $C$ depending only on 
$\B$, $p$ and $q$ such that 
\begin{equation}
\norm{X_{\gr{0},m}}_{\B,p,q}\leq C\pr{\sum_{ \gri:\norm{\gri}_\infty 
=m}\pr{\delta_{\B,p,q}\pr{\gri}}^p  }^{1/p}.
\end{equation}
\end{Lemma}
\begin{proof}
 This follows the idea of proof of Corollary~1 in \cite{MR3963881}  and 
Corollary~1.5 in \cite{MR4372657}, which is to express $X_{\gr{0},m}$ as a sum 
of a martingale difference sequence and apply a Burkholder type inequality, 
namely, Proposition~\ref{prop:moments_ordre_pq_orthomartingale_Banach} 
in the case $d=1$.
\end{proof}

\begin{proof}[Proof of Theorem~\ref{thm:loi_grands_nombres_champs_bern_rect_ps}]
 We have to show that \eqref{eq:cond_cas_bern_rectangles} implies that for 
 each positive $\eps$, 
 \begin{equation}
  \lim_{N\to\infty}
  \PP\pr{\sup_{\grn\smd\gr{1},\max\grn\geq N}\frac 
1{\abs{\grn}^{1/p}}\norm{\sum_{\gr{1}\imd\gri\imd \grn}X_{\gri}}_{\B}>\eps    
}=0.
 \end{equation}
Using the decomposition $X_{\gri}=Y_{\gri,M-1}+\sum_{m=M}^\infty X_{i,m}$,
we derive that 
\begin{multline}
\PP\pr{\sup_{\grn\smd\gr{1},\max\grn\geq N}
\frac 1{\abs{\grn}^{1/p} }\norm{\sum_{\gr{1}\imd\gri\imd 
 \gr{n}}X_{\gri}}_{\B}>\eps } 
\\
\leq 
\PP\pr{\sup_{\grn\smd\gr{1},\max\grn\geq N}
\frac 
1{\abs{\grn}^{1/p}}\norm{\sum_{\gr{1}\imd\gri\imd 
\gr{n}}Y_{\gri,M-1}}_{\B}>\frac{\eps}2 
}+\PP\pr{\sup_{\grn\smd\gr{1},\max\grn\geq N}
\frac 1{\abs{\grn}^{1/p}}\norm{\sum_{m=M}^\infty\sum_{\gr{1}\imd\gri\imd 
\gr{n}}X_{\gri,m}}_{\B}>\frac{\eps}2 }\\
\leq 
\PP\pr{\sup_{\grn\smd\gr{1},\max\grn\geq N}
\frac 
1{\abs{\grn}^{1/p}}\norm{\sum_{\gr{1}\imd\gri\imd 
\gr{n}}Y_{\gri,M-1}}_{\B}>\frac{\eps}2 
}+\PP\pr{\sum_{m=M}^\infty\sup_{\grn\smd \gr{1}} 
\frac{1}{\abs{\grn}^{1/p}}\norm{\sum_{\gr{1}\imd\gri\imd 
\gr{n}}X_{\gri,m}}_{\B}>\frac{\eps}2 }.
\end{multline}

Observe that if $\norm{\grk}_{\infty}\geq M$, then 
$P_{\grk}\pr{Y_{\gr{0},M-1}}=0$ hence by 
Theorem~\ref{thm:strong_law_large_stationary}, one 
has 
\begin{equation}
\lim_{N\to\infty}\PP\pr{\sup_{\grn\smd\gr{1},\max\grn\geq N}
\frac 
1{\abs{\grn}^{1/p}}\norm{\sum_{\gr{1}\imd\gri\imd 
\gr{n}}Y_{\gri,M-1}}_{\B}>\frac{\eps}2 }=0.
\end{equation}
Using \eqref{eq:comparaison_norm_Lp_faible} and the triangle inequality for the 
norm $\norm{\cdot}_{\B,p,w}$, it suffices to show that 
\begin{equation}\label{eq:etape_demo_LGN_rect_bernoulli}
\sum_{m=1}^\infty\norm{\sup_{\grn\smd\gr{1}}
\frac 1{\abs{\grn}^{1/p} }\norm{\sum_{\gr{1}\imd\gri\imd 
\gr{n}}X_{\gri,m}}_{\B}}_{p,w}<\infty.
\end{equation}
To do so, we use \eqref{eq:decomposition_champ_m_dependent} and 
replacing $\grn$ by $m'\gr{n'}+\gr{b}$, where 
$\gr{0}\imd\gr{b}\imd\pr{m'-1}\gr{1}$, we
derive that for each fixed $m$, 
\begin{equation*}
 \norm{\sup_{\grn\smd\gr{1}}
\frac 1{\abs{\grn}^{1/p} }\norm{\sum_{\gr{1}\imd\gri\imd 
\gr{n}}X_{\gri,m}}_{\B}}_{p,w}
\leq  \sum_{\gr{1}\imd\gr{a}\imd m'\gr{1}}
 \sum_{\gr{\delta}\in\ens{0,1}^d}
 \norm{\sup_{\gr{n'}\smd\gr{0}}\sup_{\gr{0}\imd\gr{b}\imd \pr{m'-1}\gr{1}}
\frac 1{\abs{m\gr{n'}+\gr{b} }^{1/p}  }\norm{\sum_{\gr{0}\imd\grj\imd 
\gr{n'}-\delta}X_{m'\grj+\gr{a},m}}_{\B}}_{p,w}
\end{equation*}
and by Theorem~\ref{thm:loi_forte_des_grands_nombres_orthomartingale}, that 
\begin{equation}
  \norm{\sup_{\grn\smd\gr{1}}
\frac 1{\abs{\grn}^{1/p} }\norm{\sum_{\gr{1}\imd\gri\imd 
\gr{n}}X_{\gri,m}}_{\B}}_{p,w}
\leq Cm^{d\pr{1-1/p}}\norm{X_{0,m}}_{\B,p,d-1}.
\end{equation}
We conclude by Lemma~\ref{lem:norme_Orlicz_X_0m}.
\end{proof}

\begin{proof}[Proof of Theorem~\ref{thm:loi_grands_nombres_champs_bern_carres}]
 We have to show that 
\eqref{eq:cond_cas_bern_carres} implies that 
for each positive $\eps$, 
\begin{equation}\label{eq:a_montrer_pour_LGN_Bernoulli_carre}
\lim_{N\to\infty}\PP\pr{\sup_{n\geq N} 
\frac 1{n^{d/p} }\norm{\sum_{\gr{1}\imd\gri\imd n\gr{1}}X_{\gri}}_{\B}>\eps }=0.
\end{equation}
By replacing in the beginning of proof of 
Theorem~\ref{thm:loi_grands_nombres_champs_bern_rect_ps} the suprema  
$\sup_{\grn\smd\gr{1},\max\grn\geq N}$ by $\sup_{n\geq N}$ 
and $\grn$ by $n\gr{1}$ we reduce the proof to
\begin{equation}\label{eq:etape_demo_LGN_carres_bernoulli}
\sum_{m=1}^\infty\norm{\sup_{n\geq 1} 
\frac 1{n^{d/p} }\norm{\sum_{\gr{1}\imd\gri\imd 
n\gr{1}}X_{\gri,m}}_{\B}}_{p,w}<\infty.
\end{equation}

By \eqref{eq:decomposition_champ_m_dependent} applied with $\grn=n\gr{1}$, 
and by decomposing the supremum over $n$ according to the remainder of $n$ for 
the Euclidean division by $m'$, we derive that 
\begin{equation}\label{eq:control_fct_max_etape_intermediaire}
 \norm{\sup_{n\geq 1} 
\frac 1{n^{d/p} }\norm{\sum_{\gr{1}\imd\gri\imd 
n\gr{1}}X_{\gri,m}}_{\B}}_{p,w}\leq \sum_{\gr{1}\imd\gr{a}\imd m'\gr{1}}
 \sum_{\gr{\delta}\in\ens{0,1}^d} 
 \norm{\sup_{N\geq 1} 
\pr{\frac{1}{N m'} }^{d/p}\norm{\sum_{\gr{0}\imd\grj\imd 
N\gr{1}+\gr{\delta}}X_{m'\grj+\gr{a},m}}_{\B}}_{p,w}.
\end{equation}
A minor modification of the proof of 
Theorem~\ref{thm:loi_forte_grands_nombres_carres} shows that 
\eqref{eq:controle_fonction_maximale_carres_orthomartingale} holds with 
$S_{n\gr{1}}$ replaced by $\sum_{\gr{0}\imd\grj\imd 
N\gr{1}+\gr{\delta}}X_{m'\grj+\gr{a},m}$. Therefore, by 
\eqref{eq:control_fct_max_etape_intermediaire}, 
\begin{equation}
  \norm{\sup_{n\geq 1} 
\frac 1{n^{d/p} }\norm{\sum_{\gr{1}\imd\gri\imd 
n\gr{1}}X_{\gri,m}}_{\B}}_{p,w}\leq
Cm^{d\pr{1-1/p}}\norm{X_{\gr{0},m }}_{\B,p}
\end{equation}
and \eqref{eq:etape_demo_LGN_carres_bernoulli} follows from 
Lemma~\ref{lem:norme_Orlicz_X_0m}. 

\end{proof}

\begin{proof}[Proof of 
Theorem~\ref{thm:loi_grands_nombres_champs_bern_rectangles_Lp}]
 Using similar arguments as at the beginning 
of the proof of Theorem~\ref{thm:loi_grands_nombres_champs_bern_rect_ps}, it 
suffices to prove that 
\begin{equation}
 \sum_{m=0}^\infty \sup_{\grn\smd\gr{1}}\frac 1{\abs{\grn}^{1/p}}
 \norm{\sum_{\gr{1}\imd\gri\imd\grn}X_{\gri,m}  }_{\B,p}<\infty.
\end{equation}
By \eqref{eq:decomposition_champ_m_dependent}, it is possible to 
bound $\norm{\sum_{\gr{1}\imd\gri\imd\grn}X_{\gri,m}  }_{\B,p}$ by the sum of 
$m^d$ terms, each of them being the $\norm{\cdot}_{\B,p}$-norm of an 
independent random field on a rectangle having at most 
$\prod_{\ell=1}^d\pr{n_\ell/m'+1}$ elements hence by 
Corollary~\ref{cor:moments_max_ordre_r_orthomartingale_Banach}, we infer that 
\begin{equation}
 \sup_{\grn\smd\gr{1}}\frac 1{\abs{\grn}^{1/p}}
 \norm{\sum_{\gr{1}\imd\gri\imd\grn}X_{\gri,m}  }_{\B,p}
 \leq Cm^{d\pr{1-1/p}}\norm{X_{\gr{0},m }}_{\B,p}.
\end{equation}
We conclude by Lemma~\ref{lem:norme_Orlicz_X_0m}.
\end{proof}

\begin{proof}[Proof of Corollary~\ref{cor:linear_processes}]
In view of Theorems~\ref{thm:loi_grands_nombres_champs_bern_rect_ps}, 
\ref{thm:loi_grands_nombres_champs_bern_carres} and 
\ref{thm:loi_grands_nombres_champs_bern_rectangles_Lp}, it suffices to prove 
that there exists a constant $C$ such that for each $\gri\in\Z^d$, 
\begin{equation}
 \delta_{\B,p,d-1}\pr{\gri}\leq C\norm{A_{\gri}}_{\Bca\pr{\B}}^\alpha.
\end{equation}
To do so, observe that by \eqref{eq:g_Lipschitz} and linearity of $A_{\gri}$, 
the following inequalities hold almost surely
\begin{equation}
 \norm{X_{\gri}-X_{\gri}^*}_{\B}\leq K\pr{g}\norm{\sum_{\grk\in\Z^d}
 A_k\pr{\eps_{\gri-\grk}}-\sum_{\grk\in\Z^d}
 A_k\pr{\eps^*_{\gri-\grk}}
 }_{\B}^\alpha=K\pr{g}\norm{A_{\gri}\pr{\eps_{\gr{0}}  } -
A_{\gri}\pr{\eps'_{\gr{0}}  }  }_{\B}^\alpha.
\end{equation}
Since $A_{\gri}$ is bounded, we infer that 
\begin{equation}
  \norm{X_{\gri}-X_{\gri}^*}_{\B}\leq K\pr{g} 
\norm{A_{\gri}}_{\Bca\pr{\B}}^\alpha \pr{\norm{\eps_{\gr{0}}}_{\B}^\alpha+
\norm{\eps'_{\gr{0}}}_{\B}^\alpha}
\end{equation}
and we conclude by \eqref{eq:norme_el_pq_de_puissances}.
\end{proof}

\begin{appendices}
\section{Moment inequalities for orthomartingales}\label{apA}

First recall the following, rewritten in terms of norm $\norm{\cdot}_{\B,r}$.
 \begin{Proposition}[Proposition~A.1 in \cite{giraudo2022deviation}]
 \label{prop:moments_ordre_r_orthomartingale_Banach}
 Let $\pr{\B,\norm{\cdot}_{\B}}$ be a separable $r$-smooth Banach space.
 For each $d\geq 1$, there exists a constant $C\pr{\B}$ such that  for each orthomartingale difference
 random field $\pr{D_{\gri}}_{\gri\in\Z^d}$,
 \begin{equation}\label{eq:moments_ordre_r_orthomartingale_Banach}
   \norm{\sum_{\gr{1}\imd \gri\imd\grn} D_{\gri} }_{\B,r} 
  \leq C\pr{\B,d}
\pr{\sum_{\gr{1}\imd \gri\imd\grn}
 \norm{D_{\gri}}_{\B,r}^r}^{1/r}.
 \end{equation}
 \end{Proposition}
 
Keeping in mind that an $r$-smooth Banach space is also $p$-smooth for $1<p\leq r$, we derive the following
 consequence of Proposition~\ref{prop:moments_ordre_r_orthomartingale_Banach} and iterations of Doob's inequality.
 
 \begin{Corollary} \label{cor:moments_max_ordre_r_orthomartingale_Banach}
  Let $\pr{\B,\norm{\cdot}_{\B}}$ be a separable $r$-smooth Banach space and let $1<p\leq r$. For each $d\geq 1$, there exists a constant $C\pr{\B,d,p}$ such that for each orthomartingale difference
 random field $\pr{D_{\gri}}_{\gri\in\Z^d}$,
 \begin{equation}\label{eq:moments_mac_ordre_r_orthomartingale_Banach}
     \norm{\max_{\gr{1}\imd\grn\imd\gr{N}}\norm{\sum_{\gr{1}\imd \gri\imd\grn} D_{\gri}}_{\B} }_{\R,p} 
  \leq C\pr{\B,d,p}
\pr{\sum_{\gr{1}\imd \gri\imd\gr{N}}
 \norm{D_{\gri}}_{\B,p}^p}^{1/p}.
 \end{equation}
 \end{Corollary}
 
We will also need to control the Orlicz-norm associated to the function 
$\varphi_{p,q}$ defined in \eqref{eq:def_varphi_pq} of sums of an  martingale 
difference sequence. 
 \begin{Proposition} 
 \label{prop:moments_ordre_pq_orthomartingale_Banach}
 Let $\pr{\B,\norm{\cdot}_{\B}}$ be a separable $r$-smooth Banach space.
 For each   $1<p\leq r$ and $q>0$, there exists a constant $C\pr{\B,p,q}$ such 
that for each  martingale difference sequence
 $\pr{D_{i}}_{i\geq 1}$,
\begin{equation}\label{eq:moments_ordre_rq_orthomartingale_Banach_dim1}
\norm{\sum_{i=1}^nD_i}_{\B,p,q}\leq 
C\pr{\sum_{i=1}^n\norm{D_i}_{\B,p,q}  }^{1/p}.
\end{equation}
 \end{Proposition} 
 
\begin{proof}

By Lemma~2.2 in \cite{MR3077911}, we can find a constant $K$ depending only on $\B$ and $p$ such that for each $x>0$, $\beta>1$ and $\delta\in\pr{0,\beta-1}$, 
\begin{multline}
 \PP\pr{\max_{1\leq n\leq N}\norm{\sum_{i=1}^n D_i }_{\B}>\beta x}
 \leq\pr{\frac{K\pr{\B}\delta}{\beta-\delta-1}}^r \PP\pr{\max_{1\leq n\leq 
N}\norm{\sum_{i=1}^n D_i }_{\B}>  x}\\+
\PP\pr{\max\ens{\max_{1\leq i\leq 
N}\norm{D_i}_{\B},\pr{\sum_{i=1}^N\E{\norm{D_i}_{\B}^p\mid\Fca_{i-1} } }^{1/p}  
}>\delta x  }.
\end{multline}
Proceeding as in the proof of Theorem~21.1 in \cite{MR0365692}, 
we find that there exists a constant $K\pr{\B,p,q}$ such that 
for each $\lambda$, 
\begin{multline}
\E{\varphi_{p,q}\pr{\frac{\norm{\sum_{i=1}^nD_i}_{\B}}{\lambda} }}
\\
\leq K\pr{\B,p,q}\E{\varphi_{p,q}\pr{\frac{\max\ens{\max_{1\leq i\leq 
n}\norm{D_i}_{\B},\pr{\sum_{i=1}^n\E{\norm{D_i}_{\B}^p\mid\Fca_{i-1} } }^{1/p}  
}}{\lambda} }}.
\end{multline}
Using the fact that there exists a constant $\kappa_{p,q}$ such that for each $x,y\geq 0$, $\varphi_{p,q}\pr{xy}\leq\kappa_{p,q}
\varphi_{p,q}\pr{x}\varphi_{p,q}\pr{y}$, we get that for each positive $\lambda,R$,
\begin{multline}
\E{\varphi_{p,q}\pr{\frac{\norm{\sum_{i=1}^nD_i}_{\B}}{R\lambda} }}
\\
\leq K\pr{\B,p,q}\kappa_{p,q}\varphi_{p,q}\pr{\frac 1R}\E{\varphi_{p,q}\pr{\frac{\max\ens{\max_{1\leq i\leq 
n}\norm{D_i}_{\B},\pr{\sum_{i=1}^n\E{\norm{D_i}_{\B}^p\mid\Fca_{i-1} } }^{1/p}  
}}{\lambda} }}.
\end{multline}
Take $R_0$ such that $K\pr{\B,p,q}\kappa_{p,q}\varphi_{p,q}\pr{\frac 1R_0}\leq 1$ in order to get that 
\begin{equation}
\norm{\sum_{i=1}^nD_i}_{\B,p,q}\leq R_0
\norm{\max\ens{\max_{1\leq i\leq 
n}\norm{D_i}_{\B},\pr{\sum_{i=1}^n\E{\norm{D_i}_{\B}^p\mid\Fca_{i-1} } }^{1/p}  
}}_{\B,p,q}.
\end{equation}
We derive \eqref{eq:moments_ordre_rq_orthomartingale_Banach_dim1} using
\begin{equation*}
\max\ens{\max_{1\leq i\leq 
n}\norm{D_i}_{\B},\pr{\sum_{i=1}^n\E{\norm{D_i}_{\B}^p\mid\Fca_{i-1} } }^{1/p}  
}\leq 
\pr{\sum_{i=1}^n\norm{D_i}_{\B}^p}^{1/p}+\pr{\sum_{i=1}^n\E{\norm{D_i}_{\B}^p\mid\Fca_{i-1} } }^{1/p},
\end{equation*} 
the triangle inequality and the fact that there exists a 
constant $\kappa'_{\alpha, p,q}$ such that for each non-negative random variable $Y$,
\begin{equation}\label{eq:norme_el_pq_de_puissances}
\norm{Y^{\alpha}}_{p,q}\leq \kappa'_{\alpha,p,q}\norm{Y}_{p\alpha,q}^\alpha;\quad
\norm{Y^p}_{1,q}\leq \kappa'_{p,q}\norm{Y}_{\alpha,p,q}.
\end{equation}
\end{proof} 
 \section{Bounds on series of truncated random variables}\label{apB}
 
 The truncation arguments we will use throughout the proofs 
 lead to consideration of bounds of series having the form 
 $\sum_{k=1}^\infty a_k\PP\pr{Y>b_k}$, $\sum_{k=1}^\infty a_k\E{Y\ind{Y\leq b_k}}$ or $\sum_{k=1}^\infty a_k\E{Y\ind{Y>b_k}}$ for some non-negative random variable $Y$ and some sequences $\pr{a_k}_{k\geq 1}$ and $\pr{b_k}_{k\geq 1}$.
 
\begin{Proposition}
For a non-negative random variable $Y$, $d\geq 1$, $1<p<r$, the following inequalities take place:
\begin{equation}\label{eq:series_queues_va_tronquee_carres}
\sum_{N\geq 1}2^{Nd\pr{1-r/p}}\ind{Y\leq 2^{Nd/p}}
\leq \kappa_{d,p,r}\pr{\ind{Y\leq 1}+Y^{p-r}\ind{Y>1}},
\end{equation}
\begin{equation}\label{eq:series_queues_va_non_tronquee_carres}
\sum_{N\geq 1}2^{-dN\pr{1-1/p}}\ind{Y>2^{dN/p}}\leq \kappa_{d,p}
Y^{p-1},
\end{equation}
\begin{equation}\label{eq:series_queues_va_non_tronquee_rectangles}
\sum_{k=1}^\infty 2^k k^{d-1}
\PP\pr{Y>\eps 2^{k/p}}\leq c_d\E{\varphi_{p,d-1}\pr{Y}},
\end{equation}
\begin{equation}\label{eq:series_queues_va_tronquee_rectangles}
\sum_{k=1}^\infty 2^{k\pr{1-r/p}}k^{d-1}\E{Y^r\ind{Y\leq 2^{k/p}}}
\leq \kappa_{p,r,d}\E{\varphi_{p,d-1}\pr{Y}},
\end{equation}
\begin{equation}\label{eq:series_E_Y_ind_Y>}
\sum_{k=1}^\infty 2^{k\pr{1-1/p}} k^{d-1}\E{Y\ind{Y>2^{k/p}}}\leq c_{p,d}\E{\varphi_{p,d-1}\pr{Y}}.
\end{equation}
\end{Proposition} 
 \begin{proof}
All these inequalities follow from the decompositions
\begin{equation}
\ind{Y>2^{ak}}=\sum_{j=k}^\infty \ind{2^{aj}<Y\leq 2^{a\pr{j+1}}}\mbox{ and }
\end{equation}
\begin{equation}
\ind{Y\leq 2^{ak}}=\ind{Y\leq 1}+\sum_{j=1}^k \ind{2^{a\pr{j-1}}<Y\leq 2^{aj}},
\end{equation}
 and the fact that 
\begin{equation}
\sum_{k=1}^j 2^{ka}k^{d-1}\leq \kappa_{a,d}2^jj^{d-1}.
\end{equation}
 \end{proof}
\end{appendices}
\providecommand{\bysame}{\leavevmode\hbox to3em{\hrulefill}\thinspace}
\providecommand{\MR}{\relax\ifhmode\unskip\space\fi MR }
\providecommand{\MRhref}[2]{%
  \href{http://www.ams.org/mathscinet-getitem?mr=#1}{#2}
}
\providecommand{\href}[2]{#2}

\end{document}